\documentclass[12pt]{article}

\usepackage{amsmath}
\usepackage{amsfonts}
\usepackage{amssymb}
\usepackage{amsthm}
\usepackage{makeidx}
\usepackage{graphicx, color}
\usepackage{url,hyperref}
\usepackage{array}
\usepackage{multirow}
\allowdisplaybreaks

\def\qqq{\mathbb{Q}}
\def\ccc{\mathbb{C}}
\def\zzz{\mathbb{Z}}

\newtheorem{theorem}{Theorem}[section]
\newtheorem{corollary}[theorem]{Corollary}

\newtheorem{lemma}[theorem]{Lemma}
\newtheorem{proposition}[theorem]{Proposition}

\newtheorem{example}[theorem]{Example}

\newtheorem{remark}[theorem]{Remark}

\numberwithin{equation}{section}

\begin{document}

\author{C. G. Melles\\
Department of Mathematics\\
U. S. Naval Academy\\
Annapolis, MD 21402\\
cgg@usna.edu
\and
W. D. Joyner\\
wdjoyner@gmail.com} 

\title{On $p$-ary Bent Functions \\
and Strongly Regular Graphs}


\maketitle

\begin{abstract}
Our main result is a generalized Dillon-type theorem, 
giving graph-theoretic conditions which guarantee that a $p$-ary function in an even number of variables is bent, 
for $p$ a prime number greater than $2$. 
The key condition is that the component Cayley 
graphs associated to the values of the function are 
strongly regular, and either 
all of Latin square type, or all 
of negative Latin square type.
Such a Latin or 
negative Latin square type 
bent function 
is regular or weakly regular,
respectively. 
Its dual function 
has component Cayley graphs 
with the same parameters as those 
of the original function.
We also give a criterion for bent functions involving
structure constants of association schemes. 
We prove that if a $p$-ary
function with component Cayley graphs of 
feasible degrees
determines an amorphic association scheme, 
then it is bent. 
Since amorphic association 
schemes correspond to 
strongly regular graph 
decompositions of Latin or 
negative Latin square type, 
this result is equivalent to our 
main theorem. 
We show how to construct 
bent functions from orthogonal arrays and give some examples. 
\end{abstract}

\vfill
{\parindent0pt
{\bf Keywords}: bent function, strongly regular graph, Latin square type graph, amorphic association scheme }


\newpage

\newpage
\section{Introduction}
Bent functions over a finite field can be thought of as maximally non-linear functions. They can be defined using Walsh transforms, but can also be studied using the combinatorics of their level sets, the parameters of certain associated Cayley graphs, and the algebras generated by the adjacency matrices of these graphs. 
Dillon \cite{D74} characterized bent Boolean functions 
as those whose supports form combinatorial structures 
known as difference sets of elementary Hadamard type.
An alternative and closely related characterization 
is that a Boolean function 
is bent if and only if 
its Cayley graph is  
strongly regular with parameters 
$(\nu, k, \lambda, \mu)$ 
satisfying $\lambda= \mu$
(Bernasconi, Codenotti, and VanderKam \cite{BCV01}).
These theorems do not generalize in an obvious way 
for primes $p$ greater than 2. The Cayley graphs associated with a bent $p$-ary function are not necessarily strongly regular (see, for example,
\S \ref{subsec:GF52}).

We consider $p$-ary functions over the finite field 
$GF(p)$ with $p$ elements, where $p$ is a prime number
 greater than 2. A function 
 $f \colon GF(p)^{2m} \rightarrow GF(p)$
 determines a collection of component Cayley 
 graphs corresponding to the values of $f$.
When $f$ is even, 
these graphs are undirected.
We will usually also assume that $f$ vanishes at $0$
 (a weak assumption, since 
adding a constant to a bent function results in another 
bent function).
The component Cayley graphs are regular, with degrees 
determined by the sizes of the level sets of $f$.
Our main result is a generalization 
of the theorems of Dillon and 
Bernasconi, Codenotti, and VanderKam 
in one direction.
We prove that if the component Cayley graphs of $f$ 
are all strongly regular and  
are either all of Latin square type with 
feasible degrees, 
or all of negative Latin square type with 
feasible degrees,
then $f$ is bent. 
The feasibility conditions are simply 
conditions arising from 
the possible sizes of the level sets of
a bent function.

The proof of our main theorem uses an 
expression for the Walsh transform 
of $f$ in terms of the eigenvalues 
of the component Cayley graphs. 
When the component Cayley graphs are 
strongly regular and of feasible Latin or
negative Latin square type, 
we obtain formulas for the eigenvalues and 
their multiplicities, which 
we use to calculate the values of the Walsh 
transform and to show that $f$ is bent.

As a consequence of the proof outlined above, 
we also find that functions of feasible 
Latin square type are regular, 
and functions of feasible 
negative Latin square type are $(-1)$-weakly 
regular. 
In each case, 
the component Cayley graphs of the dual function 
are strongly regular, with the same parameters as those 
of the original function. 
The proof of this duality theorem 
uses a relationship between the component functions 
of the dual function and the Fourier transforms 
of the component functions of the original function.

The papers of 
Gol'fand, Ivanov, and Klin \cite{GIK94}, 
van Dam \cite{vD03}, 
and van Dam and Muzychuk \cite{vDM10} 
describe the close relationship between 
graphs of Latin and negative Latin square type and 
amorphic association schemes.
We say that a $p$-ary function $f$ 
is {\it amorphic} if 
its level sets determine an amorphic association 
scheme.  
In a previous paper, \cite{CJMPW16}, 
we showed that for any prime $p$ greater than 2, 
there are $(p+1)!/2$ bent amorphic functions 
of two variables
with algebraic normal form homogeneous of degree $p-1$
and such that 
the level sets corresponding to nonzero elements of 
$GF(p)$ all have size $p-1$.
In this paper, we generalize part of our previous 
result by proving that 
if an even 
$p$-ary function of $2m$ variables  
with level sets of feasible sizes is amorphic, 
then it is bent.
The key to the proof is a criterion for 
a function to be bent involving sums of 
structure constants of an association scheme.
We also use a result of 
Ito, Munemasa, and Yamada \cite{IMY91}
describing the structure constants 
of an amorphic association scheme.

In light of the relationship between 
Latin and negative Latin square 
type graphs and 
amorphic association schemes described 
in \cite{GIK94}, \cite{vD03}, and \cite{vDM10}, 
we see that  
a $p$-ary function is amorphic 
if and only if its component
Cayley graphs are strongly regular 
and either all of Latin square type, 
or all of negative Latin square type.
Thus, our criterion for bent functions 
involving eigenvalues and 
our criterion involving structure 
constants lead to equivalent 
theorems,
proven by different methods.

The existence of amorphic bent $p$-ary functions 
of $2m$ variables 
follows from the existence of 
$p$-class amorphic association schemes 
with corresponding graphs of appropriate degrees.  
Such amorphic association schemes can 
be constructed from orthogonal arrays of size $(N+1) \times N^2$, 
where $N=p^m$. 
Orthogonal arrays of these dimensions exist by a construction of 
Bush \cite{Bu52}.
We describe a construction of 
amorphic bent functions of Latin square type 
from orthogonal arrays and give examples 
for $n=2$ and $4$, and $p=3$, 5, and 7.
We also give examples of amorphic bent 
functions of negative Latin square type, 
and of bent functions whose component Cayley 
graphs are not all strongly regular.

It would be interesting to find a combinatorial generalization of 
the theorems of Dillon and 
Bernasconi, Codenotti, and VanderKam
in the other direction, giving simple graph-theoretic 
properties that the component Cayley graphs of 
a bent $p$-ary function must possess. 
In the case $p=3$, Tan, Pott, and Feng \cite{TPF10} show that if
$f \colon GF(3^{2m} ) \rightarrow GF(3)$ is
an even weakly regular bent function 
with $f(0)=0$, then the component Cayley 
graphs of $f$ are strongly regular 
and either all of Latin square or all of negative 
Latin square type. Using the 
theory of quadratic residues, 
Chee, Tan, and Zhang \cite{CTZ11} and 
 Feng, Wen, Xiang, and Yin \cite{FWXY13} 
 generalize this result for $p$ a prime greater than 2.
 However, their decompositions for $p>3$ 
 are not related to the component graphs studied in this paper.  

Our paper is structured as follows.
In Section 2, we study the level 
sets of a $p$-ary function.
We use a result of Kumar, Scholtz, and Welsh 
to calculate the possible sizes of the level sets 
of an even bent $p$-ary function of $2m$ variables which 
vanishes at 0.
The component Cayley graphs 
and the corresponding component 
functions of a $p$-ary function 
are defined in Section 3. 
We explain why the eigenvalues of these graphs 
are the values of the Fourier 
transforms of the component functions. 
In Section 4, we state a criterion for 
a $p$-ary function to be bent, involving 
eigenvalues of the component Cayley graphs.
We give formulas for the eigenvalues 
of the component Cayley graphs of 
a function of feasible 
Latin or negative Latin square type in Section 5. 
We revisit the theorems of Dillon and 
Bernasconi, Codenotti, and VanderKam 
in this context.
In Section 6, we prove that $p$-ary 
functions of feasible Latin or negative Latin 
square type are bent. 
We describe their dual functions in Section 7.
In Section 8, we discuss 
$p$-ary functions that determine association schemes.
We state a structure constant criterion for such a function 
to be bent. 
Using this criterion, 
we prove that amorphic functions with 
component Cayley graphs of feasible degrees 
are bent. 
In Section 9, we describe how to construct amorphic 
bent functions of Latin square type 
from orthogonal arrays. 
Section 10 is devoted to examples, 
most of which were constructed with the 
aid of computers. We conclude with 
some questions and ideas for further study.

\section{Sizes of level sets of bent functions}
\label{sec:2}
In this section, 
we study even $p$-ary functions of 
an even number of variables, vanishing at 0.
We obtain necessary, but not sufficient, 
conditions for such a function to be bent 
by considering the sizes of its level curves.
We refer to these conditions 
as feasibility conditions.
The feasibility conditions are derived from the 
possible sizes of the level sets in the 
bent case, which we calculate using a result of 
Kumar, Scholtz, and Welsh \cite{KSW85}. 
We also state the feasibility conditions, equivalently, in 
terms of the degrees of a function's component Cayley 
graphs, in \S \ref{subsec:feasible-degrees}.

We fix, once and for all, an ordering for 
$GF(p)^n$. That ordering will be used for all vectors 
whose coordinates are indexed by $GF(p)^n$, 
all matrices whose entries are indexed by 
$GF(p)^n \times GF(p)^n$, and all vertices of associated 
component Cayley graphs defined below.

We routinely identify the elements of $GF(p)$ with 
$\{0, 1, 2, \ldots , p-1\}$.

\subsection{The Walsh transform}
Let $p$ be a prime number, and 
let $\zeta$ be the $p$th root of unity given by 
\[
\zeta = e^{\frac {2\pi i} p}.
\]
Let $n$ be a positive integer.
A $p$-ary function $f \colon GF(p)^n \to GF(p)$ 
determines a well-defined complex-valued function
$\zeta^f \colon GF(p)^n \rightarrow \ccc$.
The {\it Walsh} or {\it Walsh-Hadamard transform} of $f$ is 
defined to be the function $W_f \colon GF(p)^n \rightarrow \ccc$ given by
\begin{equation*}
	W_f(x) =
	\sum_{y \in GF(p)^n}
	\zeta^{f(y)- \langle x,y\rangle},
\end{equation*}
where $\langle \ , \ \rangle$ is the usual inner product
on $GF(p)^n$.

\subsection{The Fourier transform}
If $g \colon GF(p)^n \rightarrow \ccc$ 
is a complex-valued function on $GF(p)^n$,  
the {\it Fourier transform} of $g$ is 
the function $\hat g \colon GF(p)^n \rightarrow \ccc$ given by
\begin{equation*}
	\hat {g}(x) =
	\sum_{y \in GF(p)^n}
	g(y) \zeta^{- \langle x,y\rangle}.
\end{equation*}
Thus, the Walsh transform of $f \colon GF(p)^n \rightarrow GF(p)$ 
is the Fourier transform of $\zeta^f \colon GF(p)^n \rightarrow \ccc$.

\subsection{Bent functions}
Let $f \colon GF(p)^n \to GF(p)$ be a $p$-ary function.
We say that $f$ is {\it bent} if
	\[
	|W_f(x)|=p^{\frac n 2},
	\]
for all $x$ in $GF(p)^n$. 

We note that if $p$-ary functions $f_1$ and $f_2$ differ by a constant 
element of $GF(p)$, then $f_1$ is bent if and only if $f_2$ is 
bent, so from now on we will assume that $f(0)=0$.

We will also assume that $f$ is {\it even}, i.e., $f(x)=f(-x)$, for all $x$ in $GF(p)^n$.
When $f$ is even, the component Cayley graphs of 
$f$ are undirected (see \S \ref{subsec:cCg}).

\subsection{Level sets of a $p$-ary function}
The {\it level sets} of $f$,
for $1 \leq i \leq p-1$,
are the sets 
	\[
	D_i = \{x \in GF(p)^n \ | \ f(x) = i\}.
	\]
For consistency with our later discussions of component 
Cayley graphs and association schemes, 
we define  
	\[
	D_p = \{x \in GF(p)^n \ | \ x \neq 0 \ \text{and} \ f(x) = 0\}
	\quad \text{and} \quad D_0=\{0\}.
	\]
Thus, the level set $f^{-1}(0)$ is the union 
of $D_p$ and $\{ 0 \}$.
In our discussion of the sizes of the level sets 
of $f$, it is convenient to express our results 
in terms of the sizes of the sets $D_i$, 
because these sizes are the degrees of the component 
Cayley graphs, which we define in \S \ref{subsec:cCg} below.

\subsection{Level sets of a bent Boolean function}
\label{subsec:bentBool}
In the Boolean ($p=2$) case, the Walsh transform takes 
integer values. If $f \colon GF(2)^n \to GF(2)$ 
is bent, $n$ must be even, since $W_f(0) = \pm 2^{\frac n 2}$.
Setting $m = \frac n 2$, we find that the only possible sizes 
for $D_1$ are 
\[
|D_1| = 2^{2m-1} \pm 2^{m-1}.
\] 

\subsection{Kumar--Scholtz--Welsh theorem}
The description of the Walsh transform in the next theorem
is a key step in our 
calculation of the sizes of level sets of 
even bent $p$-ary functions.
This theorem follows directly from a result of Kumar, Scholtz, 
and Welsh \cite[Property~7]{KSW85}, but we include a proof 
for the sake of completeness.

\begin{theorem}[Kumar-Scholtz-Welsh]
\label{propn:KSW}
Suppose that $f \colon GF(p)^{2m} \rightarrow GF(p)$ is an 
even bent function, where $p$ 
is a prime number greater than $2$, 
and $m$ is a positive integer. 
Then, for every $x$ in $GF(p)^{2m}$,
	\[
	W_f(x)=\pm \zeta^j p^m
	\]
for some integer $j$ with $0 \leq j \leq p-1$.
\end{theorem}

\begin{proof} 
Fix $x$ in $GF(p)^{2m}$, and let $W=W_f(x)$.
It is sufficient to show that $p^{-m}W$ is a root of 
unity in $\qqq (\zeta)$, since the only roots 
of unity in $\qqq (\zeta)$ are those of the form 
$\pm \zeta^j$ 
(see, e.g., \cite[p.~158]{BS66} or 
\cite[Corollary 3.5.12]{C07}). We will first show that 
$p^{-m}W$ is an element of $\zzz[\zeta]$, and then 
use a theorem of Kronecker 
\cite{K1857}, which states that 
an element of $\zzz[\zeta]$, all of whose 
conjugates have magnitude 1,  
is a root of unity in $\qqq (\zeta)$
(for a more accessible source, 
see \cite[Corollary 3.3.10]{C07}). 

For $\alpha$ in $\zzz [\zeta]$,  
let $\langle \alpha \rangle$ denote the 
principal ideal generated by $\alpha$ in $\zzz [\zeta]$. 
Note that the ideal $\langle p \rangle$ 
has a factorization as $\langle p \rangle=\langle 1 - \zeta \rangle^{p-1}$, and the ideal $\langle 1-\zeta \rangle$
in $\zzz [\zeta]$ is prime 
(see, e.g., \cite[p.~157]{BS66} or \cite[Lemma 1.4]{W82}).
Let $\overline W$ be the complex conjugate of $W$.
Since $f$ is bent, $W \overline W = p^{2m}$.
Thus, the ideal in $\zzz[\zeta]$ generated 
by $W \overline W $ has a factorization into prime ideals as 
$\langle 1 - \zeta \rangle^{2m(p-1)}$. 
Suppose that
$\langle W \rangle=\langle 1-\zeta \rangle^k$ and 
$\langle \overline W \rangle = \langle 1-\zeta \rangle^\ell$ for some integers $k$ and $\ell$. 
Since $\langle 1 - \overline \zeta\rangle
= \langle 1 - \zeta\rangle$, 
we also have $\langle \overline W\rangle 
= \langle 1 - \overline \zeta\rangle^k
 = \langle 1 - \zeta\rangle^k$, 
so $k=\ell$.
Therefore, $\langle W\rangle  = \langle \overline W\rangle 
= \langle p^m \rangle$.
It follows that $W = u p^m$ for some unit $u$
of magnitude 1 in $\zzz[\zeta]$. 

The conjugates of $u$ are the images of $u$ under 
the elements of the Galois group of $\qqq (\zeta)$.
This Galois group consists of the $p-1$ automorphisms  
$\sigma_k$ of $\qqq (\zeta)$, which are 
determined by the equations  
$\sigma_k(\zeta)=\zeta^k$, 
for $1 \leq k \leq p-1$.
It is straightforward to show that 
$\sigma_k(W_f(x)) = W_{kf}(kx)$.
It can also be shown 
that $kf$ is bent, for $k$ in $\{1, 2, \dots , p-1\}$,
for example, by using the balanced derivative criterion 
of \S \ref{subsec:balanced-derivative}.
Thus, all the conjugates of $u$ under the actions 
of the maps $\sigma_k$, 
i.e., all the images $\sigma_k(u)$, 
have magnitude 1. 
It follows from the theorem of Kronecker 
\cite{K1857} mentioned above, that 
$u$ is a root of unity in $\qqq (\zeta)$. 
Therefore $u = \pm \zeta^j$ for some $j$ with 
$0 \leq j \leq p-1$.
\hfill 
\end{proof}

\subsection{Feasible sizes of level sets of $p$-ary functions}
\label{subsec:feas-size-set}
In this section, 
we calculate the possible sizes of level sets 
of even bent $p$-ary functions of 
$2m$ variables, vanishing at 0.
At the end of this section, we state feasibility 
conditions for a function to be bent, 
based on these sizes.
As a first step toward this goal, we prove the following 
corollary of the theorem of Kumar, Scholtz, and Welsh. 

\begin{corollary}
\label{cor:W_f0-real}
Suppose that $f \colon GF(p)^{2m} \to GF(p)$ is an even bent function 
such that $f(0)=0$, 
where $p$ is a prime number greater than 2, 
and $m$ is a positive integer.
Then 
\[
W_f(0)=\pm p^m.
\]
Furthermore, the level sets $D_i$, for $1 \leq i \leq p-1$, 
are all the same size, i.e., 
     \[
	| D_1| = |D_2| = \cdots =|D_{p-1}|.
	\]
\end{corollary}

\begin{proof}
Let $k_i =|D_i|= | f^{-1}(i) |$ for $1 \leq i \leq p - 1$, and 
let $k_p =|D_p|= |f^{-1}(0)| - 1$.
Notice that since $f$ is even and $f(0)=0$, $k_i$ must be an even integer, 
for $1 \leq i \leq p$.

From the definition of the Walsh transform, 
\begin{equation}
\label{eqn:W_f0}
W_f(0) = 1 + k_1 \zeta + k_2 \zeta^2 + \cdots + k_{p-1} \zeta^{p-1}  + k_p.
\end{equation}
By the result of Kumar, Scholtz, and Welsh, 
\begin{equation}
\label{eqn:Wf0KSW}
W_f(0)=\pm \zeta^j p^m,
\end{equation}
for some $j$ such that $0 \leq j \leq p-1$.
Since $1 + \zeta+ \zeta^2 + \cdots + \zeta^{p-1} = 0$,
Equation (\ref{eqn:W_f0}) can be 
rewritten as 
\begin{equation*}
W_f(0) = \sum_{i = 1}^{p-1} (k_i - 1 - k_p)\zeta^i .
\end{equation*}
The roots of unity $\zeta, \zeta^2, \ldots , \zeta^{p-1}$ are linearly independent over $\qqq$. 
It follows from Equation (\ref{eqn:Wf0KSW}) 
that if $1 \leq j \leq p-1$, then
$k_i - 1 - k_p = 0$, for $i \neq j$.
But this is impossible, since $k_i$ and $k_p$ are both even. 
Therefore $j=0$, and $W_f(0) = \pm p^m$.

Furthermore, we must have $k_i-1-k_p=-W_f(0)$ for 
$1 \leq i \leq p-1$. Therefore $k_1 = k_2 = \cdots =k_{p-1}$, i.e., 
  \[
	| D_1| = |D_2| = \cdots =|D_{p-1}|.
	\]
\hfill \end{proof}

We now calculate the possible sizes of level sets 
of even bent $p$-ary functions of 
$2m$ variables, vanishing at 0.

\begin{proposition}
\label{prop:sizeD_i}
Suppose that $f \colon GF(p)^{2m} \rightarrow GF(p)$ is an 
even bent function such that $f(0)=0$, where $p$ 
is a prime number greater than 2, and $m$ is a positive integer. 
Then 
the possible sizes of the sets 
$D_i$ 
are 
\[
|D_i| = 
(N-1) \frac N p,
\]
for $1 \leq i \leq p-1$, and
\[
|D_p| = (N-1) \left( \frac N p + 1 \right),
\]
where $N=W_f(0) = \pm p^m$.
\end{proposition}

\begin{proof} 
By Corollary \ref{cor:W_f0-real}, the Walsh transform of $f$  at $0$ is 
$W_f(0)=\pm p^m$. 
Also by Corollary \ref{cor:W_f0-real}, the sizes of the level sets $D_1, D_2, \ldots , D_{p-1}$ are all equal.  
Let $k = |D_i|$, for $1 \leq i \leq p-1$, and let 
$k_p=|D_p|$. Let $N = W_f(0)$. Since 
\[
\{0\} \cup D_1 \cup D_2 \cup \cdots \cup D_p = GF(p)^{2m},
\]
the constants $k$ and $k_p$ are related by the equation.
\begin{equation}
\label{eqn:union-D_i}
1+(p-1)k+k_p=p^{2m}=N^2.
\end{equation}
By the definition of the Walsh transform at 0, 
\begin{equation}
\label{eqn:Walsh-count}
1+k\left( \zeta+\zeta^2 + \cdots + \zeta^{p-1} \right) + k_p=N.
\end{equation}
Since $\zeta+\zeta^2 + \cdots + \zeta^{p-1}=-1$, Equation 
(\ref{eqn:Walsh-count}) can be rewritten as 
\begin{equation*}
k_p=k+N-1.
\end{equation*}
Substituting into Equation (\ref{eqn:union-D_i}) we find that 
\[
pk +N=N^2.
\]
Hence,
\[
k= \left( N - 1 \right)\frac {N} p 
\]
and 
\[
k_p=\left(N-1 \right)\left( \frac {N} p + 1 \right).
\]
\hfill
\end{proof}

\begin{remark}
{\rm
A straightforward calculation shows that if 
$f$ satisfies the hypotheses of Proposition \ref{prop:sizeD_i},
then $p$ divides the norm-squared of the \lq\lq signature"
$(|f^{-1}(0)|, |f^{-1}(1)|, \dots , |f^{-1}(p-1))|$ of the function $f$, 
i.e., $p$ divides the quantity
\[
| \{0\} \cup D_p|^2+\sum_{i=1}^{p-1} |D_i|^2.
\]
}
\end{remark}

We have described the possible sizes 
of the level sets of an even bent function of 
$2m$ variables in the Boolean case, in
\S \ref{subsec:bentBool}, and 
in the $p$-ary case, for $p$ 
a prime number greater than 2, in 
Proposition \ref{prop:sizeD_i}.
These results lead to the following feasibility conditions 
for a function to be bent.

Let $f \colon GF(p)^{2m} \rightarrow GF(p)$ 
be an even function with $f(0)=0$, 
where $p$ is a prime number, and $m$ is a positive integer. 
Let $D_i = f^{-1}(i)$ for $1 \leq i \leq p-1$, and 
let $D_p = f^{-1}(0) \setminus \{0\}$.
We say that the level sets of $f$ 
are of {\it feasible sizes} if, for $1 \leq i \leq p$, 
\begin{equation}
\label{eqn:feasible-sizes}
|D_i| = (N-1)r_i, 
\end{equation}
where 
\[
N = \pm p^m, \quad
r_i = \frac N p \ \text{for $1 \leq i \leq p-1$}, 
\quad \text{and} \quad r_p = \frac N p + 1.
\]
A function whose level sets are not of feasible sizes 
cannot be bent. 
In the next section, we will 
state these feasibility conditions, equivalently, in 
terms of the degrees of a function's component Cayley 
graphs (see \S \ref{subsec:feasible-degrees}).

\section{Cayley graphs of $p$-ary functions}
In this section, we describe 
a collection of regular graphs 
$\{\Gamma_1, \Gamma_2, \dots , \Gamma_p\}$, 
the component Cayley graphs, 
associated to a $p$-ary function $f$.
We define the 
component functions $f_i$
of $f$ to be the indicator functions of the sets $D_i$
described above.
The eigenvalues of the adjacency matrix of $\Gamma_i$ 
are the values of the 
Fourier transform of $f_i$.
The adjacency matrices of the component Cayley graphs commute.

\subsection{Component functions of a $p$-ary function}
Suppose that $f \colon GF(p)^n \rightarrow GF(p)$ is 
a $p$-ary function.
Recall that we define $D_i = f^{-1}(i)$, for $1 \leq i \leq p-1$, 
and $D_p = f^{-1}(0)\setminus \{0\}$.
The {\it component functions}
$f_i \colon GF(p)^n \rightarrow \ccc$ 
of $f$ are defined to be the indicator functions of the 
sets $D_i$, given by
\begin{equation}
\label{eqn:defn-fi}
f_i(x)=
\begin{cases}
1 &\qquad \text{if $x \in D_i$,}\\
0 &\qquad \text{otherwise,}
\end{cases}
\end{equation}
for $1 \leq i \leq p$.

\subsection{Component Cayley graphs of a $p$-ary function}
\label{subsec:cCg}
Let $f \colon GF(p)^n \rightarrow GF(p)$ be an 
even function with $f(0)=0$, where $p$ 
is a prime number greater than 2, 
and $n$ is a positive integer. 
The function $f$ determines a graph decomposition 
$\{ \Gamma_1, \Gamma_2, \dots , \Gamma_p\}$  
 of the complete graph on the vertex set $GF(p)^n$.
For $1 \leq i \leq p-1$, 
there is an edge in $\Gamma_i$ 
between distinct vertices $x$ and $y$ in $GF(p)^n$ 
if $f(x-y)=i$, i.e., if $x-y \in D_i$.
There is an edge in $\Gamma_p$
between distinct vertices $x$ in $y$ in $GF(p)^n$
if $f(x-y)=0$, i.e., if $x-y \in D_p$.
Note that these graphs may be considered undirected, since $f$ is even, 
so $f(x-y)=f(y-x)$.
The graph $\Gamma_i$ is the Cayley graph 
of the pair $(GF(p)^n, D_i)$.
We refer to the graphs $\Gamma_i$ as the 
{\it component Cayley graphs} or simply 
the {\it Cayley graphs} of $f$.
We can also regard $\Gamma_i$ as the Cayley graph of 
the component function $f_i$. The graph $\Gamma_i$ is regular 
of degree $| D_i|$, i.e., every vertex is of degree 
$|D_i|$.

For example, let $f \colon GF(3)^2 \rightarrow GF(3)$ be given 
by $f(x_0,x_1)=-x_0^2+x_1^2$. 
The component Cayley 
graphs $\Gamma_1$, $\Gamma_2$, and $\Gamma_3$ of $f$ 
are shown in Figure \ref{fig:3-aryLST}.

\begin{figure}[ht]
\begin{center}
\begin{tabular}{ccc}
\includegraphics[scale=0.30]{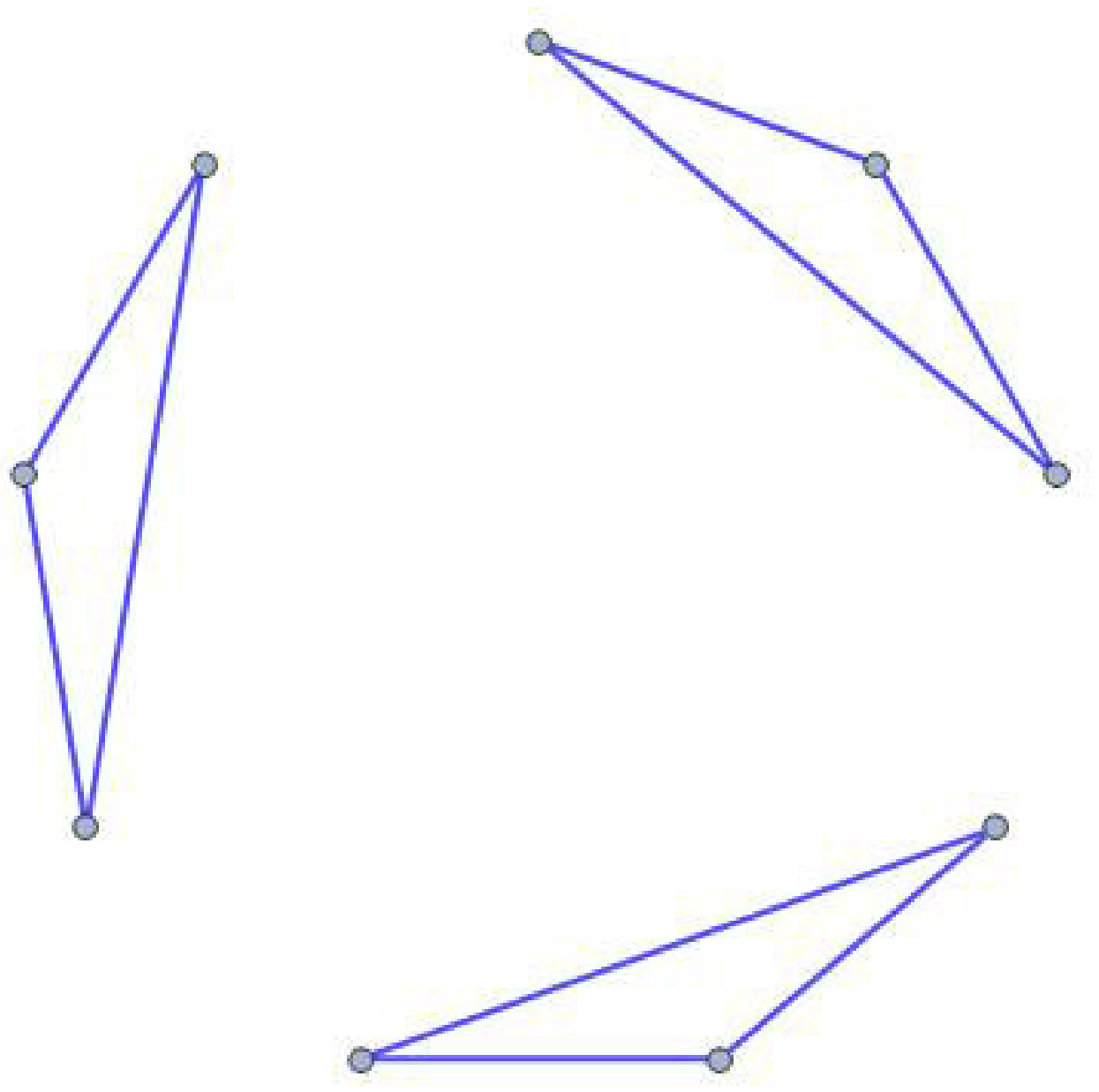}&
\includegraphics[scale=0.30]{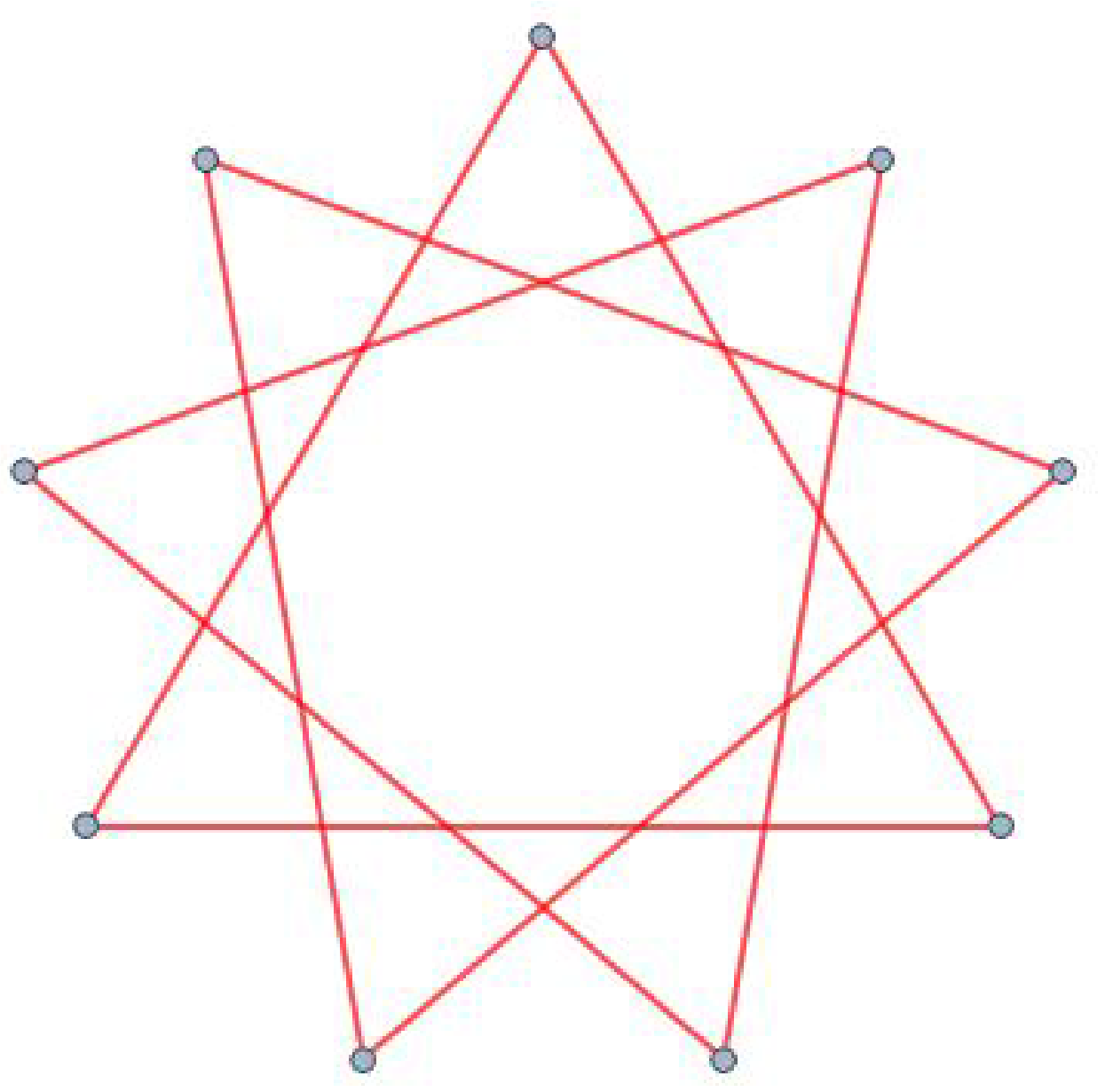}&
\includegraphics[scale=0.30]{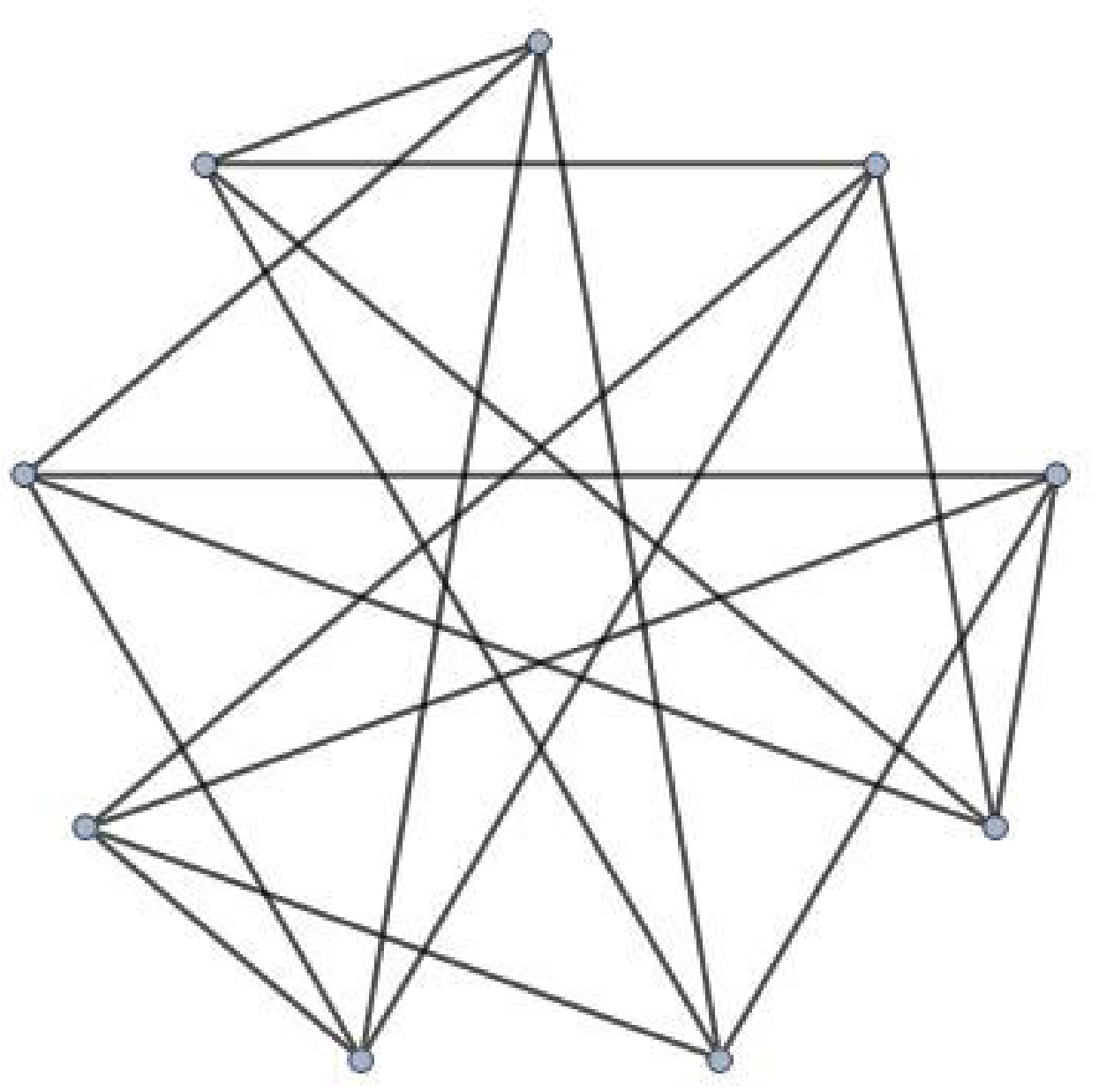}
\end{tabular}
\end{center}
\caption{{The Cayley graphs of the $3$-ary function $-x_0^2+x_1^2$.} }
\label{fig:3-aryLST}
\end{figure}

\subsection{Feasible degrees of component Cayley graphs}
\label{subsec:feasible-degrees}
Suppose that $f \colon GF(p)^{2m} \rightarrow GF(p)$ is an 
even function such that $f(0)=0$, where $p$ 
is a prime number greater than 2, and $m$ is a positive integer. 
In Equation (\ref{eqn:feasible-sizes})
of Section \ref{subsec:feas-size-set},
we described the feasible sizes of the 
level sets of $f$. 
If the level sets of $f$ are not of these sizes, 
$f$ cannot be bent.
We now restate these feasibility conditions 
in terms of degrees of graphs.
We say that the component Cayley graphs $\Gamma_i$
of $f$ are of {\it feasible degrees} 
if the degrees of these graphs correspond to 
the feasible sizes of level sets, i.e., 
\[
\text{degree}(\Gamma_i)=(N-1)r_i,
\] 
where 
\[
N = \pm p^m, \quad
r_i = \frac N p 
\ \text{for $1 \leq i \leq p-1$}, \quad
\text{and} \quad
r_p = \frac N p + 1.
\]  
If the degrees of the graphs $\Gamma_i$ 
are not of these sizes, the function $f$ 
cannot be bent.

\subsection{Adjacency matrices}
Let $\Gamma$ be a matrix with vertex set $V$ of size $\nu$. 
The {\it adjacency matrix} of $\Gamma$, with respect to 
a fixed ordering of the vertices,
is the $\nu \times \nu$ matrix $A$ whose 
rows and columns are indexed by the elements of $V$, such that 
\begin{equation*}
A_{xy}=
\begin{cases}
1 &\qquad \text{if $x \neq y$ and $(x,y)$ is an edge of $\Gamma$,}\\
0 &\qquad \text{otherwise.}
\end{cases}
\end{equation*}

Let $f \colon GF(p)^n \rightarrow GF(p)$ be an even function such that $f(0)=0$.
Let $A(i)$ be the adjacency matrix of the component Cayley graph $\Gamma_i$ 
with respect to the ordering of $GF(p)^n$ 
fixed in \S \ref{sec:2}.
We will show below that the matrices 
$A(1), A(2), \dots ,A(p)$ commute, 
since they share a common basis of eigenvectors.

\subsection{Hadamard vectors}
\label{subsec:Had}
Let $\nu=p^n$.
For each vector $x$ in $GF(p)^n$, we define a vector $h(x)$ in 
$\ccc^\nu$, using the same fixed ordering of $GF(p)^n$ 
as in \S \ref{sec:2}, by
\[
h(x)_y=\zeta^{-\langle x,y \rangle}.
\]
We call the vectors $h(x)$ {\it generalized Hadamard} or 
simply {\it Hadamard} vectors.
The vector $h(0)$ is the all 1's vector.
By the following lemma, the remaining vectors 
$h(x)$, where $x \neq 0$, span the 
subspace of $\ccc^\nu$ orthogonal to $h(0)$.

\begin{lemma}
\label{lem:Hadvec}
The $\nu$ Hadamard vectors $h(x)$ in $\ccc^\nu$ are 
orthogonal and linearly independent over $\ccc$. 
\end{lemma}

\begin{proof}
Let $H$ be the matrix whose columns are the Hadamard vectors $h(x)$, for $x$ in $GF(p)^n$.
It is straightforward to show that  
$H \overline H^t = \nu I$,
where $\nu=p^n$, and $I$ is the $\nu \times \nu$ 
identity matrix.
\hfill
\end{proof}

The matrix $H$ whose columns are the vectors $h(x)$ is sometimes 
called a {\it generalized Hadamard} or {\it Butson} matrix.

\subsection{Eigenvalues corresponding to Hadamard vectors}
\label{subsec:eig-Had-h}
Suppose that $\Gamma$ is a graph with vertex set $V$ of 
size $\nu$,
and $A$ is the adjacency matrix of $\Gamma$.
The set of eigenvalues of $A$ is called the {\it spectrum} of the 
graph $\Gamma$.
We sometimes refer to the eigenvalues of $A$ as eigenvalues of $\Gamma$.

Let $f \colon GF(p)^n \rightarrow GF(p)$ 
be a even function with $f(0)=0$, 
with component Cayley graphs 
$\Gamma_i$ and corresponding adjacency matrices $A(i)$.
We will show that the Hadamard vectors $h(x)$ of 
 Lemma \ref{lem:Hadvec}
form a basis of common eigenvectors of the matrices 
$A(i)$ over $\ccc$, and 
the values of the 
Fourier transforms $\hat {f_i}$ of the component 
functions $f_i$ are the eigenvalues 
of these matrices.

\begin{lemma}
\label{lemma:lamfi}
The Hadamard vector $h(x)$ is an 
eigenvector of $A(i)$ corresponding to the 
eigenvalue $\hat f_i(x)$, for each $x$ in $GF(p)^n$,
and for $1 \leq i \leq p$.
\end{lemma}

\begin{proof}
The entry in position $y$ in the product $A(i)h(x)$ is 
\begin{align*}
\left( A(i)h(x) \right)_y 
&= \sum_t A(i)_{yt}h(x)_t\\
&=\sum_t f_i(t-y)\zeta^{-\langle x, t \rangle}\\
&=\sum_z f_i(z)\zeta^{-\langle x, z+y \rangle}\\
&= \left(\sum_z f_i(z) \zeta^{-\langle x, z \rangle} \right)
      \zeta^{-\langle x, y \rangle}\\
&= \hat {f_i}(x)  h(x)_y.
\end{align*}
Thus, 
\[
A(i) h(x) =  \hat{f_i}(x) h(x),
\]
i.e., the vector $h(x)$ is an 
eigenvector of $A(i)$ corresponding to the 
eigenvalue $\hat f_i(x)$.
\hfill
\end{proof}

As an immediate corollary of the previous lemma, we 
see that the adjacency matrices $A(i)$ commute.

\begin{corollary}
If $f \colon GF(p)^n \rightarrow GF(p)$ is 
an even $p$-ary function with $f(0)=0$, 
the adjacency matrices $A(1), A(2), \dots ,A(p)$ of 
the component Cayley graphs of $f$ commute.
\end{corollary}

We will often use the notation $\lambda_i(x)$ 
to denote the eigenvalue of $A(i)$ corresponding 
to the Hadamard eigenvector $h(x)$. Thus,
\begin{equation}
\label{eqn:lambdaiofx}
\lambda_i(x)=\hat {f_i}(x),
\end{equation}
for $1 \leq i \leq p$ and for all $x$ in $GF(p)^n$.

\section{Eigenvalue criterion for bent functions}
In this section, we characterize even bent $p$-ary 
functions in terms of eigenvalues of their component 
Cayley graphs.

\begin{proposition}
\label{propn:Walsheigenvalues}
Let $f \colon GF(p)^n \rightarrow GF(p)$ 
be an even function such that $f(0)=0$, where $p$ is is 
a prime number, and $n$ is a positive integer.
Then the Walsh transform of $f$ satisfies 
\[
W_f(x) = 1 + \sum_{i=1}^p \zeta^i \lambda_i(x),
\]
for all $x$ in $GF(p)^n$, where $\lambda_i(x)$ 
is the eigenvalue from Equation (\ref{eqn:lambdaiofx}) 
above.
\end{proposition}

\begin{proof}
Recall that we defined $D_i = f^{-1}(i)$ for $1 \leq i \leq p-1$, 
and $D_p = f^{-1}(0) \setminus \{0\}$.
As above, we denote by  
$f_i$ the component function of $f$ defined in 
Equation (\ref{eqn:defn-fi}), and by $\hat{f_i}$ its 
Fourier transform.
The Walsh transform of $f$ can be 
written in terms of the Fourier transforms of 
the functions $f_i$ as
\begin{equation*}
\begin{aligned}
W_f(x) 
& =\sum_{y \in GF(p)^n} \zeta^{f(y)} \zeta^{- \langle y, x \rangle }\\
& = \zeta^0 
 + \sum_{i=1}^{p-1} \sum_{y \in D_i} \zeta^i \zeta^{-\langle y , x\rangle}
  + \sum_{y \in D_p} \zeta^{- \langle y ,x\rangle}  \\
&= 1 + \sum_{i=1}^{p-1} \zeta^i 
 \sum_{y \in GF(p)^n} f_i(y) \zeta^{-\langle y , x \rangle}
  + \sum_{y \in GF(p)^n} f_p(y) \zeta^{- \langle y , x \rangle} \\
&= 1 + \zeta \hat {f_1}(x) + \zeta^2 \hat {f_2}(x) + \cdots + \zeta^{p-1}\hat  {f_{p-1}}(x)
+ \hat {f_p}(x).
\end{aligned}
\end{equation*}
Since $\lambda_i(x)=\hat {f_i}(x)$, 
(see Equation (\ref{eqn:lambdaiofx}) above),
this completes the proof.
\hfill
\end{proof}

From the previous result,
we obtain the following characterization of an 
even bent $p$-ary function 
in terms of eigenvalues of its component Cayley graphs.

\begin{proposition}
\label{propn:eigenvaluecharacterization}
Let $f \colon GF(p)^n \rightarrow GF(p)$ 
be an even function such that $f(0)=0$, where $p$ is  
a prime number, and $n$ is a positive integer.
Then $f$ is bent if and only if 
\[
|1+\sum_{i=1}^p \zeta^i \lambda_i(x) | = p^{\frac n 2}
\]
for all $x$ in $GF(p)^n$, where $\lambda_i(x)$ 
is the eigenvalue from Equation (\ref{eqn:lambdaiofx}) 
above.
\end{proposition}

\section{Feasible Latin and negative Latin square type functions}
In this section, we consider even $p$-ary 
functions whose component Cayley graphs are 
strongly regular and  
either all of Latin square type with 
feasible degrees, 
or all of negative Latin square type with 
feasible degrees.
We describe the eigenvalues 
of the component Cayley graphs of such functions.
In order to illustrate how our main result is related 
to the theorems of Dillon and 
Bernasconi, Codenotti, and VanderKam, 
we recast their theorems in this context.
We begin with some background material 
on strongly regular graphs and graphs of 
Latin and negative Latin square type 
(for further details 
see, for example, Godsil and Royle \cite[Chapter~10]{GR01}). 

\subsection{Strongly regular graphs}
Let $\Gamma$ be a $k$-regular graph on $\nu$ vertices 
(every vertex has degree $k$). 
The graph $\Gamma$ is called 
{\it strongly regular} if 
there exist nonnegative integers $\lambda$ and $\mu$ such that 
if $x$ and $y$ are neighbors in $\Gamma$, 
there are $\lambda$ common neighbors of $x$ and $y$, and
if $x$ and $y$ are not neighbors in $\Gamma$, there are $\mu$ 
common neighbors of $x$ and $y$.
The constants $(\nu , k , \lambda , \mu)$ are called the
{\it parameters} of the graph $\Gamma$.

A strongly regular graph on the vertex set $GF(p)^n$
with parameters $(v, k, \lambda, \mu)$ 
corresponds to a symmetric partial difference set with parameters $(v, k, \lambda, \mu)$ 
(see, for example, \cite[Chapter 6]{JM17}). 
Thus, many of our statements 
about strongly regular graphs 
could be rephrased in terms 
of symmetric partial difference sets.

\subsection{Eigenvalues of strongly regular graphs}
The eigenvalues of the adjacency matrix of 
a strongly regular graph and their multiplicities can 
be expressed in terms of the parameters of the graph 
by the following well-known formulas.
A strongly regular graph $\Gamma$ 
with parameters $(\nu,k,\lambda,\mu)$
has eigenvalues $k$, $\theta$, and $\tau$, where 
the eigenvector $k$ corresponds to the all $1$'s vector,
\[
\theta = \frac { \lambda - \mu + \sqrt 
{(\lambda - \mu)^2 + 4 (k - \mu)}} 2,
\]
and 
\[
\tau = \frac { \lambda - \mu - \sqrt 
{(\lambda - \mu)^2 + 4 (k - \mu)}} 2.
\]
From these equations we see that 
\begin{equation}
\label{eqn:thetatau0}
\theta +\tau = \lambda - \mu \qquad \text{and}
\qquad \theta \tau = -k + \mu.
\end{equation}
The multiplicities of $\theta$ and $\tau$ on the space of 
vectors in $\ccc^\nu$ orthogonal to the all 1's vector are
given by 
\[
m_\theta = \frac {(\nu - 1) \tau + k}{\tau - \theta}
\qquad \text{and} \qquad 
m_\tau = \frac {(\nu - 1)\theta + k}{\theta - \tau}.
\]

\subsection{Latin and negative Latin square type graphs}
\label{subsec:LSTNLST}
We say that a strongly regular graph $\Gamma$ is of {\it Latin
square type} if there exist integers $N>0$ and $r>0$ such that 
the parameters of $\Gamma$ are 
\begin{equation}
\label{eqn:LSTdefn}
(\nu,k,\lambda, \mu)=(N^2, (N-1)r, N+r^2-3r, r^2-r).
\end{equation}
A strongly regular graph $\Gamma$ is of {\it negative Latin
square type} if there exist integers $N<0$ and $r<0$ such that 
the parameters of $\Gamma$ are given by Equation (\ref{eqn:LSTdefn}).

If $\Gamma$ is a strongly regular graph of Latin square type 
then the eigenvalues of $\Gamma$ are  
\begin{equation}
\label{eqn:thetatau1}
k=(N-1)r, \qquad
\theta= N - r, \qquad \text{and} \qquad \tau= -r,
\end{equation}
where $\theta$ has multiplicity $m_\theta=(N-1)r$ 
on the subspace of $\ccc^\nu$ orthogonal to the all $1$'s vector,
and $\tau$ has multiplicity $m_\tau=(N-1)(N-r+1)$.

If $\Gamma$ is a strongly regular graph of 
negative Latin square type, then the eigenvalues of $\Gamma$ are
\begin{equation}
\label{eqn:thetatau2}
k=(N-1)r, \qquad 
\theta= - r, \qquad \text{and} \qquad \tau= N-r,
\end{equation}
where $\theta$ has multiplicity $m_\theta=(N-1)(N-r+1)$ 
on the subspace of $\ccc^\nu$ orthogonal to the all $1$'s vector,
and $\tau$ has multiplicity $m_\tau=(N-1)r$.

If $N= \pm p^m$ and $r= \frac N p$, the eigenvalues 
$k$, $\theta$, and $\tau$ are distinct, except in the 
case that $m=1$ and $N = p$. In this case $r=1$, 
so $\theta = k = p-1$, $\lambda=p-2$, $\mu=0$, and $\tau=-1$.
In this case, when there are only two distinct 
eigenvalues, the graph $\Gamma$ is not connected, 
and consists of $p$ copies of the complete graph $K_p$.

\subsection{Latin and negative Latin square type functions}
\label{defn:fLSTNLST}
Let $f \colon GF(p)^n \rightarrow GF(p)$ be 
a $p$-ary function, where $p$ is prime number,
and $n$ is a positive integer. 
We say that $f$ is of 
{\it Latin square type} if 
$f$ is even with $f(0)=0$, and 
its component Cayley graphs 
are all strongly regular and of 
Latin square type.
Similarly, we say that $f$ is of 
{\it negative Latin square type} if 
$f$ is even with $f(0)=0$, and 
its component Cayley graphs 
are all strongly regular and of 
negative Latin square type.

We will sometimes use the abbreviations LST and NLST 
for Latin square type and negative Latin square type, 
respectively.

\subsection{Feasible LST and NLST functions}
\label{subsec:LSTFdefn}
Let $f \colon GF(p)^{2m} \rightarrow GF(p)$ be 
a $p$-ary function, where $p$ is a prime number, 
and $m$ is a positive integer. 
We say that $f$ is of 
{\it feasible Latin square type} if 
it is of Latin square type 
and the parameters of the component 
Cayley graph $\Gamma_i$ are 
\begin{equation}
\label{eqn:LSTdefn2a}
(\nu,k,\lambda, \mu)=(N^2, (N-1)r_i, N+r_i^2-3r_i, r_i^2-r_i),
\end{equation}
where $N = p^m$, 
$r_i = \frac N p$ 
for $1 \leq i \leq p-1$, 
and $r_p = \frac N p + 1$. 
Similarly, we say that $f$ is of 
{\it feasible negative Latin square type} if 
it is of negative Latin square type and 
the parameters of the component Cayley graphs 
are given by 
Equation (\ref{eqn:LSTdefn2a}), 
where $N = -p^m$, 
$r_i = \frac N p$ 
for $1 \leq i \leq p-1$, 
and $r_p = \frac N p + 1$. 

\begin{remark}
\label{rk:no_NLST_m1p5}
{\rm
If $m=1$, there are no functions of feasible 
negative Latin square type for $p \geq 5$. 
If there were such a function, the 
formula above would give a value of 
$\lambda = -p+4$, 
which is impossible since $\lambda \geq 0$.
}
\end{remark}

\subsection{Eigenvalues of feasible LST and NLST graphs}
From the formulas of the 
previous two sections, we can calculate the 
eigenvalues of the component Cayley graphs of $f$ 
and their multiplicities, in the case that $f$ 
is of feasible Latin or negative Latin square type.
In this section, we capture a more subtle feature of 
how these eigenvalues interact, 
which is key to proving our main result and 
the subsequent duality theorem. 

Let $f \colon GF(p)^{2m} \rightarrow GF(p)$ be an 
even function such that $f(0)=0$, where $p$ is a prime 
number greater than 2,
and $m$ is a positive integer. 
Let $A(i)$ be the adjacency matrix of 
the component Cayley graph $\Gamma_i$.
Recall that we denote by $\lambda_i(x)$ the eigenvalue 
of $A(i)$ corresponding to the Hadamard vector 
$h(x)$ defined in \S \ref{subsec:Had}.

\begin{proposition}
\label{propn:distinguished-eigenvalue}
Let $f \colon GF(p)^{2m} \rightarrow GF(p)$ be a 
$p$-ary function, where $p$ is a prime number 
greater than 2,
and $m$ is a positive integer. 
Suppose that $f$ is of feasible Latin 
or negative Latin square type. 
Then for each nonzero $x$ in $GF(p)^{2m}$, there 
exists a unique distinguished index $j$ in $\{1, 2, \dots , p-1,p\}$ 
such that 
\[
\lambda_j(x) = N - r_j,
\]
while for all the remaining values of $i$ 
such that $1 \leq i \leq p$ and $i \neq j$, 
\[
\lambda_i(x) = - r_i.
\]
\end{proposition}

\begin{proof}
Let $\nu=p^{2m}$, 
let $I$ be the $\nu \times \nu$ identity matrix,
and let $J$ be the $\nu \times \nu$ all 1's matrix.
Recall that for $x$ in $GF(p)^n$, the Hadamard vectors $h(x)$ are orthogonal 
vectors in $\ccc^\nu$, 
and $h(0)$ is the all 1's vector.
Thus, if $x$ is a nonzero point in $GF(p)^n$, 
then $Jh(x)=0$.
The adjacency matrices $A(i)$ of the component graphs $\Gamma_i$ of $f$ 
satisfy
\[
I + \sum_{i=1}^p A(i)=J.
\]
Multiplying on the right by $h(x)$, where $x \neq 0$, gives
\begin{equation*}
h(x)+\sum_{i=1}^p A(i) h(x)
=\left(1 +\sum_{i=1}^p \lambda_i(x)\right)h(x)
=0,
\end{equation*}
It follows that if $x \neq 0$,
\begin{equation}
\label{eqn:sumlambda}
1+\sum_{i=1}^p \lambda_i(x) = 0.
\end{equation}

Since the component graphs $\Gamma_i$  
are of Latin or negative Latin square type 
with $N=\pm p^m$,  
$r_i = \frac N p$ for $1 \leq i \leq p-1$, and 
$r_p=\frac N p +1$, 
the eigenvalues of $\Gamma_i$ 
are $k_i=(N-1)r_i$, $N-r_i$, and $-r_i$.
When $x \neq 0$, the 
eigenvalue $\lambda_i(x)$ must take the value $N-r_i$ or $-r_i$, 
for $1 \leq i \leq p$.

Let $a$ be the number of the eigenvalues in the 
set $\{ \lambda_i(x) \ | \ 
1 \leq i \leq p-1 \}$ which take the value $N - r_i$.
Similarly, let $b=1$ if $\lambda_p(x) = N - r_p$,
and let $b=0$ otherwise. We wish to show that one of the numbers 
$a$ and $b$ is 1 and the other is 0.
Let $r$ be the common value of $r_i$ for $1 \leq i \leq p-1$, \
i.e., $r = \frac N p$. Then $r_p = r+1$.
Substituting into Equation (\ref{eqn:sumlambda}) we obtain
\begin{align*}
0
&=
1+a( N - r) + (p-1-a)(-r)
+ b (N-r_p) + (1-b)(-r_p)\\
&= 1+a( N - r) + (p-1-a)(-r)
+ b (N-r-1) + (1-b)(-r-1)\\
&=1+ (a+b)N + p(-r) -1\\
&= (a+b)N-N.
\end{align*}
Thus, $a+b=1$, so there is exactly one index $j$, 
with $1 \leq j \leq p$, 
such that $\lambda_j(x)=N-r_j$.
For the remaining values of $i \neq j$, $\lambda_i(x)=-r_i$.
\hfill
\end{proof}

\subsection{Dillon and Bernasconi--Codenotti--VanderKam\\ theorems}
Suppose that $f \colon GF(2)^{2m} \to GF(2)$
is a Boolean function such that $f(0)=0$, 
where $m$ is a positive integer.
We say that the Cayley graph $\Gamma$ of $f$ is
the component Cayley graph $\Gamma_1$.
Recall from \S \ref{subsec:bentBool} 
that if $f$ is bent, then 
the only possible sizes 
for the set $D_1=f^{-1}(1)$  
(and hence the only possible degrees of
the Cayley graph $\Gamma$) are 
$|D_1| = 2^{2m-1} \pm 2^{m-1}$.
These are the feasible degrees of the 
Cayley graph of a Boolean function of 
$2m$ variables that vanishes at 0.
Dillon's criterion \cite{D74} for 
bent Boolean functions was 
stated in the language of 
difference sets.
We state an essentially equivalent 
version in terms 
of strongly regular graphs.

\begin{theorem}[Dillon]
Let $f \colon GF(2)^{2m} \to GF(2)$ be a function with $f(0)=0$. 
Then $f$ is bent if and only if its 
Cayley graph 
is strongly regular 
of feasible Latin or negative Latin square type.
\end{theorem}

Bernasconi, Codenotti, and VanderKam \cite{BCV01} proved that 
a function $f \colon GF(2)^{2m} \to GF(2)$ with $f(0)=0$ 
is bent if and only if the Cayley graph $\Gamma$ 
of $f$ is 
strongly regular with parameters $(2^{2m}, k, \lambda, \lambda)$, 
for some $\lambda$, where $k = |D_1|$.
From the discussion above, we see that 
$k = 2^{2m-1} \pm 2^{m-1}$ and 
$\lambda=2^{2m-2} \pm 2^{m-1}$.

In the next section, we show  
that the theorems of Dillon and
Bernasconi, Codenotti, and VanderKam can be generalized 
in one direction (Theorem \ref{thm:main1}). 
In \S \ref{subsec:GF52}, 
we give 
examples 
to show that converse of 
Theorem \ref{thm:main1} does not hold, 
since the component Cayley graphs of a 
$p$-ary bent function are not 
necessarily strongly regular.

\section{Feasible Latin and negative Latin square type functions are bent}
We now prove our main result.

\begin{theorem}
\label{thm:main1}
Let $f \colon GF(p)^{2m} \rightarrow GF(p)$ be an 
even function such that $f(0)=0$, where $p$ is a 
prime number greater than 2,
and $m$ is a positive integer. 
If the component Cayley graphs of $f$ 
are all strongly regular and  
are either all of feasible Latin square type, 
or all of feasible negative Latin square type,
then $f$ is bent.
\end{theorem}

\begin{proof}
By the hypotheses of the theorem, 
the component Cayley graphs $\Gamma_i$ of $f$ are 
strongly regular with parameters 
\[
\left(N^2,(N-1)r_i,N+r_i^2-3r_i,r_i^2-r_i \right),
\]
where $N=p^m$ if the graphs are all of feasible 
Latin square type, 
$N = - p^m$ if the graphs are all of 
feasible negative Latin square type, 
$r_i = \frac N p$, for $1 \leq i \leq p-1$, and $r_p=\frac N p+1$.

In order to show that $f$ is bent, we wish to show 
that the magnitude of the Walsh transform $W_f(x)$ 
is $|N| = p^m $, for each $x$ in $GF(p)^{2m}$.

The Walsh transform at $x=0$ is easily calculated, 
by counting the number of times $f$ takes each 
value $i$, as 
\begin{align*}
W_f(0)&= \sum_{y \in GF(p)^{2m}} \zeta^{f(y)}\\
&=1+ \sum_{i=1}^{p-1} \zeta^i (N-1)r_i+(N-1)r_p\\
&=1 +(N-1)\left(\frac N p\right) \sum_{i=1}^{p-1} \zeta^i
   + (N-1) \left(\frac N p +1\right)\\
&=1-(N-1)\left(\frac N p\right)
     + (N-1) \left(\frac N p +1\right)\\
&=N.
\end{align*}
Thus, $W_f(0)= \pm p^m$.

We will now apply the characterization of bent functions in 
terms of eigenvalues (Proposition \ref{propn:eigenvaluecharacterization})
to show that $|W_f(x)|=p^m $ for 
all nonzero $x$ in $GF(p)^{2m}$.
Recall from Proposition \ref{propn:distinguished-eigenvalue} 
that for each nonzero $x$ in $GF(p)^{2m}$, there 
exists a unique distinguished value of 
$j$ in $\{1, 2, \dots , p-1,p\}$ 
such that 
$\lambda_j(x) = N - r_j$,
while for all the remaining values of $i$ 
such that $1 \leq i \leq p$ and $i \neq j$, 
we have
$\lambda_i(x) = - r_i$.
Continuing with this notation 
and using Proposition \ref{propn:Walsheigenvalues}, we find that 
\begin{align*}
W_f(x) &= 1 + \sum_{i=1}^p \zeta^i \lambda_i(x)\\
&= 1 + \zeta^j (N - r_j) + \sum_{i \neq j} \zeta^i (-r_i)\\
& = 1 + \zeta^j N + \sum_{i=1}^p \zeta^i (-r_i)\\
&= 1 + \zeta^j N 
    - \left(\frac N p\right)\sum_{i=1}^{p-1} \zeta^i 
          - \left( \frac N p +1\right)\\
&=1+\zeta^j N +\left(\frac N p\right)  - \left( \frac N p +1\right)\\ 
&=\zeta^j N.        
\end{align*}
Therefore $f$ is bent.
\hfill
\end{proof}

As an immediate consequence of the proof of the theorem above, 
we obtain the following corollary.

\begin{corollary}
\label{cor:Wfzetaj}
Let $f \colon GF(p)^{2m} \rightarrow GF(p)$ be a 
$p$-ary function, where $p$ is a prime number greater than 2,
and $m$ is a positive integer. 
Suppose that $f$ is of feasible Latin 
or negative Latin square type.
Then $W_f(0) = p^m$ in the feasible Latin square type case, 
and $W_f(0)=-p^m$ in the feasible negative Latin square type case.
Furthermore, if $x \neq 0$, then
\[
W_f(x) = \zeta^j W_f(0),
\]
where $j$ is the distinguished index 
described in Proposition \ref{propn:distinguished-eigenvalue}
such that $\lambda_j(x)$ has the form $N - r_j$.
\end{corollary}

This corollary gives us a dual $p$-ary function 
$f^*$, satisfying 
\[
W_f(x) = \zeta^{f^*(x)} W_f(0).
\]
The properties of $f^*$ are 
described in the next section.

\section{Dual functions}
In this section, we prove that a bent function $f$ 
of feasible Latin or negative Latin square type 
(as defined in \S \ref{subsec:LSTFdefn}) has a
dual $f^*$ whose component Cayley graphs 
have the same parameters as those of $f$.
The main idea of the proof is to relate the component functions 
$f_i^*$ of $f^*$ to the component functions 
$f_i$ of $f$ by means of the equation
\begin{equation}
\label{eqn:fi*}
f_i^*(x)=\frac 1 N \hat {f_i}(x) + \frac {r_i} N 
- r_i \delta_0(x),
\end{equation}
where $\hat {f_i}$ is the Fourier transform of $f_i$, and 
$\delta_0$ is the delta function centered at 0
(see Equation (\ref{defn:delta})).
Since the eigenvalues of the component Cayley graphs 
of a $p$-ary function are 
given by the Fourier transforms of the component functions, 
Equation (\ref{eqn:fi*}) allows us to calculate the eigenvalues of the 
component Cayley graphs of $f^*$ and show that these 
graphs are strongly regular with the 
desired parameters.

\subsection{Regular and weakly regular bent functions}
A bent function $f \colon GF(p)^n \to GF(p)$ is said to 
be {\it regular} if there exists a {\it dual function}
$f^* \colon GF(p)^n \rightarrow GF(p)$ 
such that
\[
W_f(x) =\zeta^{f^*(x)}  p^{\frac n 2}  
\]
for all $x$ in $GF(p)^n$.
Similarly, a bent function $f \colon GF(p)^n \to GF(p)$ is said to 
be {\it weakly regular} or {\it $\mu$-weakly regular}
if there exists a constant $\mu$ in $\ccc$ with magnitude 1
and a {\it dual function}
or {\it $\mu$-weakly regular dual function}
$f^* \colon GF(p)^n \rightarrow GF(p)$ 
such that
\[
W_f(x) = \mu \zeta^{f^*(x)}  p^{\frac n 2}
\]
for all $x$ in $GF(p)^n$.
It is known that the dual $f^*$ 
of a regular or weakly regular bent function $f$
is also bent. If $f$ is regular, so is $f^*$, and if $f$ is weakly regular, 
so is $f^*$.  
If $f$ is an even function, then $f^*$ 
is also even.
See \cite[\S 6.4]{JM17} for 
further background on duality.

\subsection{Regularity and feasible LST and NLST functions}
We show that a function
$f \colon GF(p)^{2m} \rightarrow GF(p)$ 
of feasible Latin or negative Latin square type 
(as defined in \S \ref{subsec:LSTFdefn})
is regular or weakly regular, respectively.

Recall that in the feasible Latin or 
negative Latin square case, the
parameters of the component 
Cayley graph $\Gamma_i$ 
are 
\begin{equation*}
(\nu,k,\lambda, \mu)=(N^2, (N-1)r_i, N+r_i^2-3r_i, r_i^2-r_i),
\end{equation*}
where 
\begin{equation}
\label{eqn:LSTdefn2}
N = \pm p^m, \quad
r_i = \frac N p \ \text{for $1 \leq i \leq p-1$}, 
\quad \text{and} \quad r_p = \frac N p + 1.
\end{equation} 
When $N=p^m$, the graph $\Gamma_i$ is 
of Latin square type, and when 
$N=-p^m$, it is of 
negative Latin square type.

Note that 
in the case $p=3$ and $m=1$, 
it is possible to have 
a strongly regular graph decomposition 
of $GF(p)^{2m}$ that is both 
Latin and negative Latin square type
 (see Example \ref{ex:negpos}),
but only the negative Latin square type 
graph decomposition is feasible.

As above, we denote 
the eigenvalue of $\Gamma_i$ corresponding 
to the Hadamard eigenvector $h(x)$
by $\lambda_i(x)$ 
(see Sections \ref{subsec:Had} and 
\ref{subsec:eig-Had-h}). 

\begin{proposition}
\label{propn:f^*defn}
Let $f \colon GF(p)^{2m} \rightarrow GF(p)$ be a 
$p$-ary function, where $p$ is a prime number greater than 2,
and $m$ is a positive integer.
If $f$ is of feasible Latin square type,
then $f$ is a regular bent function.
If $f$ is of feasible negative Latin square type,
then $f$ is a $(-1)$-weakly regular bent function.
\end{proposition}

\begin{proof}
The function $f$ is bent, by Theorem \ref{thm:main1}.
We define $f^*(0)=0$.
By Proposition 
\ref{propn:distinguished-eigenvalue},
for every nonzero $x$ in $GF(p)^{2m}$, 
there is a unique 
distinguished index $j$ in $\{1, 2, \dots , p\}$ such 
that $\lambda_j(x)=N-r_j$. 
If $j \neq p$, we define $f^*(x)=j$, and if $j = p$, 
we define $f^*(x)=0$.
By Corollary \ref{cor:Wfzetaj},
$W_f(x) =\zeta^{f^*(x)}  p^m$ 
in the Latin square type case, 
and $W_f(x) =-\zeta^{f^*(x)}  p^m$ 
in the negative Latin square type case.
Thus, $f$ is regular in the 
Latin square type case and 
$(-1)$-weakly regular in the negative 
Latin square type case, 
and $f^*$ is the dual of $f$.
\hfill
\end{proof}

\subsection{Level sets of dual functions}
We describe the level 
sets of the dual function of $f$, 
when $f$ is of feasible 
Latin or negative Latin square type,
using Proposition 
\ref{propn:distinguished-eigenvalue},
on distinguished indices of eigenvalues.

Suppose that $f \colon GF(p)^{2m} \to GF(p)$ 
is an even bent function with $f(0)=0$
such that $f$ is regular or weakly regular.
Let $f^* \colon GF(p)^{2m} \to GF(p)$ be 
the dual function of $f$.
Recall that we define $D_i = f^{-1}(i)$, for 
$1 \leq i \leq p-1$, and $D_p = f^{-1}(0) \setminus \{0\}$.
Similarly we define the corresponding sets 
for the dual function:
\begin{equation}
\label{defn:Di*}
D^*_i=(f^*)^{-1}(i), 
\ \text{for $1 \leq i \leq p-1$, and} 
\ D^*_p = (f^*)^{-1}(0) \setminus\{0\}.
\end{equation}

Recall also that the eigenvalues of the adjacency matrix 
$A(i)$ of the component Cayley graph 
 $\Gamma_i$ of $f$ are the values 
 $\hat f_i (x)$ of the Fourier transform of $f_i$
 for $x$ in $GF(p)^{2m}$. 
 More concisely, 
\begin{equation*}
\lambda_i(x)=\hat {f_i}(x),
\end{equation*}
for $1 \leq i \leq p$ and for all $x$ in $GF(p)^{2m}$.
If the graphs $\Gamma_i$ are strongly regular, 
and all of feasible Latin square type 
or all of feasible negative Latin square type, 
then the eigenvalues of $A(i)$ are 
$k_i = (N-1)r_i$, $N - r_i$, and $-r_i$, 
where $N$ and $r_i$ are as in Equation 
(\ref{eqn:LSTdefn2}).

\begin{proposition}
\label{prop:Di*size}
Let $f \colon GF(p)^{2m} \rightarrow GF(p)$ be a 
feasible Latin or negative Latin square type 
$p$-ary function, where $p$ is a prime number greater than 2,
and $m$ is a positive integer.
Then the set $D_i^*$ is given by 
\[
D_i^* = \{ x \in GF(p)^{2m} \setminus\{0\} 
\ | \ \hat {f_i}(x) = N - r_i\},
\]
for $1 \leq i \leq p$.
Furthermore, the cardinality $k_i^*$ of $D_i^*$ is 
\[
k_i^* = |D_i^*| = (N-1)r_i = |D_i|=k_i.
\]
\end{proposition}

\begin{proof}
The description of $D_i^*$ follows directly 
from the description of the dual function $f^*$ 
in the proof of Proposition \ref{propn:f^*defn} 
(noting that $\lambda_i(x) = \hat {f_i}(x)$).
For every nonzero $x$ in $GF(p)^{2m}$, 
there is a unique 
distinguished index $i$ in $\{1, 2, \dots , p\}$ such 
that $\lambda_i(x)=N-r_i$, 
by Proposition 
\ref{propn:distinguished-eigenvalue}.
The multiplicity of $N-r_i$ as an eigenvalue 
of $\Gamma_i$ on the orthogonal 
complement of the all 1's vector 
is given by $(N-1)r_i$
(see \S\ref{subsec:LSTNLST}).
\hfill
\end{proof}

\subsection{Component functions of dual functions}
We formulate an expression 
for the $i$th component function of a dual 
function $f^*$, in terms of the Fourier transform 
of the $i$th component function of the 
original function $f$. 

The $i$th component function of $f^*$ is 
the function $f^*_i \colon GF(p)^{2m} \rightarrow \ccc$ 
given by
\[
f_i^*(x) = 
\begin{cases}
1 &\qquad \text{ if $x \in D_i^*$,}\\
0 & \qquad \text{otherwise,}
\end{cases}
\]
for $1 \leq i \leq p$. 

The following result is an immediate corollary of 
Proposition \ref{prop:Di*size}.

\begin{corollary}
\label{cor:Di*f_hat}
Let $f \colon GF(p)^{2m} \rightarrow GF(p)$ be a 
feasible Latin or negative Latin square type 
$p$-ary function, where $p$ is a prime number greater than 2,
and $m$ is a positive integer.
Then $f_i^*(0)=0$, and 
\begin{equation*}
f_i^*(x)=1 \quad 
\text{if and only if} \quad
\hat {f_i}(x)=N-r_i,
\end{equation*}
for $x$ in $GF(p)^{2m} \setminus\{0\}$
and $1 \leq i \leq p$, 
where $N$ and $r_i$ are as in 
Equation (\ref{eqn:LSTdefn2}).
\end{corollary}

In order to relate
the component functions of the dual 
 $f^*$ to the component functions of $f$, we 
introduce a delta function and its Fourier transform.
We define the {\it delta function at $0$ on $GF(p)^{2m}$} 
to be the function $\delta_0 \colon GF(p)^{2m} \to \ccc$ 
given by 
\begin{equation}
\label{defn:delta}
\delta_0(x) = 
\begin{cases}
1 & \qquad \text{if $x=0$,}\\
0 & \qquad \text{otherwise.}
\end{cases}
\end{equation}
We denote by $\iota$ the constant function
$\iota \colon GF(p)^{2m} \to \ccc$ 
given by $\iota(x) = 1$,
for $x$ in $GF(p)^{2m}$.

The next lemma follows directly from the definition of 
the Fourier transform.  

\begin{lemma}
\label{lem:Fouriotadelta}
The Fourier transform of the delta function at $0$ 
on $GF(p)^{2m}$ is given by 
\[
\hat \delta_0 = \iota,
\]
where $\iota$ is the function 
defined above, with 
constant value 1 on $GF(p)^{2m}$.
The Fourier transform of $\iota$ is given by 
\[
\hat \iota = p^{2m} \delta_0.
\]
\end{lemma}

The following proposition 
expresses the component functions of the 
dual $f^*$ in terms of the 
Fourier transforms of the component 
functions of $f$.

\begin{proposition}
\label{propn:fi*2}
Let $f \colon GF(p)^{2m} \rightarrow GF(p)$ be a 
feasible Latin or negative Latin square type 
$p$-ary function, where $p$ is a prime number greater than 2,
and $m$ is a positive integer.
Then the $i$th component function $f_i^*$ of 
the dual of $f$ satisfies
\begin{equation}
\label{eqn:fi*2}
f_i^*(x)=\frac 1 N \hat {f_i}(x) + \frac {r_i} N 
- r_i \delta_0(x),
\end{equation}
for $1 \leq i \leq p$,
where $N$ and $r_i$ are as in Equation 
(\ref{eqn:LSTdefn2}), and $\hat {f_i}$ is 
the Fourier transform of the $i$th 
component function of $f$.
\end{proposition}

\begin{proof}
Recall that $\hat {f_i}(x)=\lambda_i(x)$, 
where $\lambda_i(x)$ is the 
eigenvalue of the adjacency 
matrix of $\Gamma_i$ corresponding 
to the Hadamard eigenvector $h(x)$. 
Thus, by Equations (\ref{eqn:thetatau1}) and 
(\ref{eqn:thetatau2}) 
of \S \ref{subsec:LSTNLST}, 
$\hat{f_i}(0)=(N-1)r_i$, 
and $\hat{f_i}(x)$ equals either 
$N-r_i$ or $-r_i$ for $x \neq 0$.
By Proposition \ref{prop:Di*size}, for $x \neq 0$, 
$\hat{f_i}(x)=N-r_i$ if and 
only if $x \in D_i^*$, 
where $D_i^*$ is as in 
Equation (\ref{defn:Di*}). 
Therefore,
\begin{align*}
\frac 1 N \hat {f_i}(x) + \frac {r_i} N 
- r_i \delta_0(x)
&=
\begin{cases}
\frac 1 N (N-1)r_i + \frac {r_i} N - r_i 
&\ \text{if $x=0$,}\\
\frac 1 N (N-r_i) + \frac {r_i} N - 0
&\ \text{if $x\in D_i^*$,}\\
\frac 1 N (-r_i) + \frac {r_i} N - 0
&\ \text{otherwise,}
\end{cases}\\
&=
\begin{cases}
1 & \ \text{if $x \in D_i^*$}\\
0 &\ \text{otherwise,}
\end{cases}\\
&=f_i(x).
\end{align*}
\hfill
\end{proof}

\subsection{Eigenvalues for dual functions}
We calculate the eigenvalues 
of the component Cayley graphs 
of dual functions, 
using the Fourier transforms of the 
corresponding component functions described 
in Proposition \ref{propn:fi*2} above.

We first review some basic properties of 
inverse Fourier transforms, which 
we will need in the proof of the next proposition.
Recall that the Fourier transform of 
a function $g \colon GF(p)^n \to \ccc$ 
is the function $\hat g \colon GF(p)^n \to \ccc$
given by 
\[
\hat g (x) = \sum_{y \in GF(p)^n} g(y) 
\zeta^{- \langle x,y \rangle}.
\]
The {\it inverse Fourier transform} of a function $h \colon 
GF(p)^n \to \ccc$ is the function 
$\check h \colon GF(p)^n \to \ccc$ 
given by 
\[
\check h (x) =\frac 1 {p^n} \sum_{y \in GF(p)^n} h(y) 
\zeta^{\langle x,y \rangle}.
\]
The Fourier transform and its inverse satisfy
\[
\check{\hat g} (x) = g(x).
\]
If $g$ is an even function, 
i.e., if $g(-x)=g(x)$ for all $x$ in $GF(p)^n$, 
then 
\begin{align*}
\check g (x) 
&=\frac 1 {p^n} \hat g (x).
\end{align*}
Thus, if $g$ is even, 
\begin{equation}
\label{eqn:ghathat}
\hat{\hat g} (x) = p^n g(x).
\end{equation}

The next proposition is dual to 
Proposition \ref{propn:fi*2} 
in the sense that it 
expresses the Fourier transforms 
of the component functions of the 
dual function
$f^*$ in terms of the component 
functions of the original function $f$.

\begin{proposition}
\label{propn:Fourierfi*}
Let $f \colon GF(p)^{2m} \rightarrow GF(p)$ be a 
feasible Latin or negative Latin square type 
$p$-ary function, where $p$ is a prime number greater than 2,
and $m$ is a positive integer.
Then the Fourier transform 
of the $i$th component function $f_i^*$ of 
the dual of $f$ satisfies
\[
(f_i^*)^\wedge (x) 
=
N f_i(x)+N r_i \delta_0(x) - r_i,
\]
for $1 \leq i \leq p$,
where $N$ and $r_i$ are as in 
Equation (\ref{eqn:LSTdefn2}).
\end{proposition}

\begin{proof}
Taking the Fourier transform of each term 
of Equation (\ref{eqn:fi*2}) of Proposition 
\ref{propn:fi*2}, we obtain
\begin{equation*}
( f_i^*)^\wedge (x) 
= \frac 1 N \hat{\hat {f_i}} (x) + 
\frac {r_i} N \hat \iota(x) - r_i \hat {\delta_0} (x),
\end{equation*}
where $\iota$ is the 
constant function with value 1 on $GF(p)^{2m}$.
By Equation (\ref{eqn:ghathat}), 
$\hat{\hat {f_i}} (x) =p^{2m}f_i(x) = N^2f_i(x)$.
The Fourier transforms of $\iota$ and $\delta_0$ 
are given by Lemma \ref{lem:Fouriotadelta} 
as $\hat \iota(x)= N^2 \delta_0(x)$
and $\hat {\delta_0}(x)=1$ . 
Therefore 
\begin{equation*}
( f_i^*)^\wedge (x) 
 = N f_i(x) + N r_i \delta_0(x) -r_i.
\end{equation*}
\hfill
\end{proof}

From Proposition \ref{propn:Fourierfi*},
we obtain the eigenvalues of the 
component Cayley graphs of the dual function.

\begin{corollary}
\label{cor:Gamma*}
Let $f \colon GF(p)^{2m} \rightarrow GF(p)$ be a 
feasible Latin or negative Latin square type 
$p$-ary function, where $p$ is a prime number greater than 2,
and $m$ is a positive integer.
Then the eigenvalues of the component Cayley graph $\Gamma_i^*$ 
of the dual of $f$ 
and their multiplicities are the same as those of 
$\Gamma_i$, for $1 \leq i \leq p$.
Specifically,
\begin{enumerate}
\item the eigenvalue $k^*_i=(N-1)r_i$ 
corresponds to the all 1's eigenvector,  
\item the eigenvalue $N-r_i$ 
occurs with multiplicity 
$(N-1)r_i$ on the vector space orthogonal to the 
all 1's vector, and
\item the eigenvalue $-r_i$ 
occurs with multiplicity 
$(N-1)(N+1-r_i)$ 
on the vector space orthogonal to the 
all 1's vector,
\end{enumerate}
where $N$ and $r_i$ are as in 
Equation (\ref{eqn:LSTdefn2}).
\end{corollary}

\begin{proof}
Let $f_i^*$ be the $i$th component function of the dual of $f$.
The eigenvalues of $\Gamma_i^*$ are the values 
$(f_i^*)^\wedge(x)$ of the Fourier transform of $f_i^*$, 
for $x$ in $GF(p)^{2m}$.
By Proposition \ref{propn:Fourierfi*}, 
the Fourier transform of $f_i^*$ is given by 
\[
(f_i^*)^\wedge (x) 
=
N f_i(x)+N r_i \delta_0(x) - r_i.
\]
Thus,
\begin{align*}
(f_i^*)^\wedge (x) 
&=
\begin{cases}
N r_i  - r_i 
& \qquad \text{ if $x=0$,}\\
N  - r_i 
& \qquad \text{ if $x \in D_i$,}\\
 - r_i 
& \qquad \text{ if $x \in GF(p)^{2m} \setminus 
\left( \{0\} \cup D_i \right)$,}
\end{cases}
\end{align*}
so $\Gamma^*_i$ has the same eigenvalues as $\Gamma_i$.
The multiplicity of $N-r_i$ on 
the vector space orthogonal to the 
all 1's vector is $|D_i| = (N-1)r_i$. 
The multiplicity of $- r_i$ on 
the vector space orthogonal to the 
all 1's vector is
\[
|GF(p)^{2m} \setminus \left( \{0\} \cup D_i  \right) |
= N^2 - 1 - (N-1) r_i = (N-1)(N+1-r_i),
\] 
Therefore, the multiplicities of the eigenvalues of $\Gamma_i^*$ 
are the same as those of $\Gamma_i$. 
\hfill
\end{proof}

\subsection{Duality theorem}
In this section we prove that the dual of a feasible 
Latin or negative Latin square type function 
is also a feasible Latin or, respectively, 
negative Latin square type function.

\begin{theorem}
\label{thm:duals}
Let $f \colon GF(p)^{2m} \rightarrow GF(p)$ be an 
even function such that $f(0)=0$, where $p$ is a prime 
number greater than 2,
and $m$ is a positive integer.
Suppose that the component Cayley graphs 
$\Gamma_i$ of $f$
are all strongly regular and  
are either all of feasible Latin square type, 
or all of feasible negative Latin square type.
Then the component Cayley graphs $\Gamma_i^*$ 
of the dual function $f^*$ 
are also all strongly regular, and 
the parameters of  $\Gamma_i^*$ are the same as 
the parameters of $\Gamma_i$. 
\end{theorem}

\begin{proof}
Recall that a simple regular graph which is not complete or edgeless 
and which has exactly two distinct eigenvalues corresponding to 
eigenvectors orthogonal to the all 1's vector must be a 
strongly regular graph (see, for example, Brouwer and Haemers,
\cite[Theorem 9.1.2]{BH11}).
Thus, it follows directly from Corollary \ref{cor:Gamma*}
that the component Cayley 
graphs $\Gamma_i^*$ are all strongly regular,
for $1 \leq i \leq p$, since each has at most 3 eigenvalues.
Of the parameters $(\nu,k_i^*, \lambda_i^*, \mu_i^*)$ 
for $\Gamma_i^*$, the parameters $\nu=N^2$ and $k_i^* = (N-1)r_i$ 
are known, (where, as above, $N =  \pm p^m$, 
$r_i = \frac N p$ 
for $1 \leq i \leq p-1$, 
and $r_p = \frac N p + 1$).   
We will calculate the parameters $\mu_i^*$ and $\lambda_i^*$ 
from the eigenvalues of $\Gamma_i^*$, 
using Corollary \ref{cor:Gamma*}.
We denote the two distinct eigenvalues of $\Gamma^*_i$ 
on the vector space orthogonal to the all 1's vector 
by $\theta_i$ and $\tau_i$ (where by convention
$\theta_i >\tau_i$, although this order 
is not needed here). By Corollary \ref{cor:Gamma*}, 
one of these two eigenvalues is $N-r_i$, and the 
other is $-r_i$. Therefore, 
by Equation (\ref{eqn:thetatau0}),  
\[
\mu_i^* = k_i^* + \theta_i^*\tau_i^* 
= r_i^2-r_i,
\]
and 
\[
\lambda_i^* = \mu_i^* + \theta_i^*+\tau_i^*
= N + r_i^2-3r_i.
\]
It follows that the graphs $\Gamma_i^*$ are either 
all of feasible Latin square type or all of 
feasible negative Latin square type.
\hfill
\end{proof}

\section{Amorphic bent functions}
In this section we consider even $p$-ary 
functions which determine association 
schemes.
We give a criterion 
for such a function to be bent, in terms of 
structure constants of its association 
scheme.
We show that a function which 
determines an amorphic association 
scheme and 
whose component Cayley graphs are 
of feasible degrees
must be bent.

It is well-known that a $p$-ary function 
$f \colon GF(p)^n \rightarrow GF(p)$ is bent if and 
only if the derivative functions given by 
$\mathcal D_b f(x) = f(x+b)-f(x)$ are balanced, 
for all nonzero $b$ in $GF(p)^n$ 
(i.e., $\mathcal D_b f$ takes all values equally often).
In the case that $f$ determines an 
association scheme,
we show that the number of times 
$\mathcal D_b f$ takes each value in 
$GF(p)$ can be expressed 
in a natural way in terms of 
structure constants of the association 
scheme. 
The structure constants of amorphic 
association schemes were 
described by Ito, Munemasa, and Yamada \cite{IMY91}. 
By summing the appropriate structure constants, 
we show that 
$\mathcal D_b f$ is balanced 
if $f$ is amorphic.

\subsection{Association schemes}
\label{sec:defn-assoc-scheme}
Let $V$ be a finite set.
A {\it binary relation} $R$ on $V$ is a subset of $V \times V$.
The {\it dual} of a relation $R$ is the set
$R^* = \{(x,y)\in V\times V\ |\ (y,x)\in R\}$.

Let $\{R_0, R_1, \dots, R_p\}$ be a set of disjoint
binary relations on $V$ whose union is $V \times V$, 
such that 
$R_0=\{ (x,x)\in V\times V\ |\ x\in V\}$,
and such that for each $i$ there is a $j$ 
for which $R_i^*=R_j$.
For $0 \leq i, j, k \leq p$ and for $(x,y)$ in $R_k$, 
let 
	\[ \rho_{ij}^k(x,y) = |\{z \in V \ | \ \text{$(x,z) \in R_i$ and $(z,y) \in R_j$}\}	|.
	\]
We say that the collection $(V, R_0, R_1, \dots , R_p)$ 
forms a {\it $p$-class association scheme} if the numbers $\rho_{ij}^k(x,y)$ 
are independent of 
which pair $(x,y)$ we choose in $R_k$
(hence depend only on $i$, $j$, and $k$).
The numbers $\rho_{ij}^k$ are called the 
{\it structure constants}
or {\it intersection numbers} of the association 
scheme.
If, in addition, $R_i^*=R_i$ for all $i$, then we say 
that the association scheme is {\it symmetric}.
A symmetric association scheme determines a collection 
of undirected graphs 
$\{ \Gamma_1, \Gamma_2 , \dots , \Gamma_p\}$ 
on the vertex set $V$.

We are interested in the case in which $V=GF(p)^n$ 
and the relations $R_i$ correspond to the component 
Cayley graphs $\Gamma_i$ of a $p$-ary function.
Let $f \colon GF(p)^n \rightarrow GF(p)$ be an 
even function such that $f(0)=0$, where $p$ 
is a prime number greater than 2, and $n$ is a positive integer.
Associated with $f$ is a set of binary relations 
$\{ R_0, R_1, \dots , R_{p-1},R_p\}$ on $V$ 
given by 
	\[
	R_0=\{ (x,x)\in V\times V\ |\ x\in V\},
	\]
	\[
	R_i = \{ (x,y) \in V \times V \ | \ f(x-y)=i\}
	\]
for $1 \leq i \leq p-1$,
and
	\[
	R_p=\{ (x,y) \in V \times V \ | \ \text {$f(x-y) = 0$ and $x \neq y$} \}.
	\]
Note that, by the assumption that $f$ is even, these relations are 
all self-dual. 
Furthermore, 
\begin{equation*}
(x,y) \in R_i \quad \text{if and only if} \quad
x-y \in D_i,
\end{equation*}
where $D_0 = \{0\}$, $D_i = f^{-1}(i)$ for $1 \leq i \leq p-1$, 
and $D_p = f^{-1}(0) \setminus \{0\}$.
The relations $R_i$, for $1 \leq i \leq p$, 
correspond to the component Cayley graphs $\Gamma_i$ of $f$.

Recall that we denote by $A(i)$ the adjacency matrix 
of the component Cayley graph $\Gamma_i$ of $f$.
Let $A(0)=I$, the $\nu \times \nu$ identity matrix, 
where $\nu = p^n$.
The matrices $A(0), A(1) , \dots , A(p)$ 
can also be thought of as 
the adjacency matrices of the relations 
$R_0, R_1, \dots , R_p$.
The sum of these adjacency matrices is the 
all 1's matrix $J$.
The condition that $f$ determines a $p$-class symmetric 
association scheme with structure constants $\rho_{ij}^k$
is equivalent to the condition that there 
exist nonnegative integers $\rho_{ij}^k$ 
such that 
\begin{equation*}
A(i)A(j) = \sum_{k=0}^p \rho_{ij}^k A(k)
\end{equation*}
for $0 \leq i, j, k \leq p$.
In this case, the matrices $A(0),A(1), \dots , A(p)$ 
generate a Bose-Mesner algebra (see, e.g., 
\cite[\S 6.7.1]{JM17}).
The structure constants can be calculated 
from the adjacency matrices  
by the following formula 
(see \cite[Chapter 17]{CvL80}):
\begin{equation}
\label{eqn:rhoijkformula}
\rho^k_{ij} 
= \left(\dfrac{1}{p^n|D_k|}\right) tr(A(i)A(j)A(k)),
\end{equation}
for $1 \leq i, j, k \leq p$,
where $tr$ denotes the matrix trace, 
and the sets $D_k$ are as above.

\subsection{Amorphic association schemes and functions}
Let $V$ be a finite set, and 
let $\mathcal R = \{R_0, R_1, \dots, R_p\}$ be a set of disjoint
binary relations on $V$ whose union is $V \times V$.
A set of disjoint binary relations $\mathcal T = \{T_0, T_1, \dots , T_m\}$ 
whose union is $V \times V$ is called a {\it fusion} of 
$\mathcal R$ if each $T_i$ is a union of 
elements of $\mathcal R$. 
An association scheme $(V, R_0, R_1, \dots , R_p)$ 
is called {\it amorphic} if for each fusion $\mathcal T$ of $\mathcal R$, 
the collection $(V, T_0, T_1, \dots , T_m)$ is also 
an association scheme. 
A 2-class association scheme is trivially amorphic.

Consider an even function 
$f \colon GF(p)^n \rightarrow GF(p)$  
such that $f(0)=0$.
Let $V=GF(p)^n$, and let $\{R_0, R_1, \dots , R_p\}$ 
be the binary relations determined by $f$, as described above.
We call $f$ {\it amorphic} if 
$(V, R_0, R_1, \dots , R_p)$ is an amorphic 
association scheme. 

\subsection{van Dam and Gol'fand--Ivanov--Klin theorems}
\label{subsec:vDam}
There is a close relationship between amorphic 
association schemes and strongly regular graphs 
of Latin and negative Latin square type.

A theorem of Gol'fand, Ivanov, and Klin
from \cite{GIK94} 
(which we learned of from van Dam and Muzychuk 
\cite{vDM10}), states that the graphs 
determined by a $p$-class amorphic association scheme, 
with $p\geq 3$, are all 
strongly regular, 
and are either all of Latin square type, 
or all of negative Latin square type. 

Van Dam \cite[Theorem 3]{vD03} 
proved the converse: a decomposition 
of a complete graph 
into strongly regular graphs, 
all of Latin square type, or 
all of negative Latin square type, 
determines an amorphic association scheme.
Thus, $p$-ary functions whose 
component Cayley graphs are 
strongly regular and 
all of Latin square type or 
all of negative Latin square type 
are amorphic functions.

\subsection{Balanced derivative criterion for bent functions}
\label{subsec:balanced-derivative}
Let $f \colon GF(p)^n \rightarrow GF(p)$ 
be a $p$-ary function.
The {\it derivative} function $\mathcal D_bf \colon GF(p)^n \rightarrow GF(p)$ is defined 
by 
	\[
	\mathcal D_bf(x) = f(x+b) - f(x).
	\]
If $f$ is linear, then $\mathcal D_bf$ is constant.
The following result, which is well-known (see, 
e.g., \cite[Proposition 6.3.9]{JM17}),
implies that bent functions are in 
some sense maximally non-linear.

\begin{proposition}
\label{propn:bent-balanced-derivative}
The function $f$ is bent if and only if $\mathcal D_bf$ is balanced for all 
$b \neq 0$, i.e., if $\mathcal D_bf$ takes each value in 
$GF(p)$ equally often.
\end{proposition}

\subsection{Structure constant criterion for bent functions}
In this section, we consider even $p$-ary 
functions whose component Cayley graphs 
determine symmetric $p$-class association schemes. 
We state a criterion for such a function 
to be bent, involving sums of 
structure constants of these schemes. 

Let $f \colon GF(p)^n \rightarrow GF(p)$ 
be an even function with $f(0)=0$.
Suppose that $f$
determines an association scheme with
structure constants $\rho_{ij}^k$, 
as described in \S \ref{sec:defn-assoc-scheme}.
Let $D_i = f^{-1}(i)$, for $1 \leq i \leq p-1$, and 
let $D_p = f^{-1}(0)\setminus \{0\}$.

\begin{proposition}
\label{propn:balanced} 
Suppose that $b \in D_i$, for some $i$ such that $1 \leq i \leq p$.
The number of times $\mathcal D_bf$ takes the value $j$, 
for $1 \leq j \leq p-1$, is 
\begin{equation*}
	\left( \sum_{k=0}^{p} \rho_{j+k \ (\text{mod}\ p), k}^i \right) + \rho_{p,p-j}^i.
\end{equation*}
\end{proposition}

\begin{proof}
Suppose that $x \in D_k$ and $x+b \in D_m$, 
for some $m$ and $k$ such that $0 \leq m,k \leq p$.
Then $f(x+b) - f(x)=j$ if and only if 
$m-k \ (\text{mod} \ p )= j$. 
Either $m = j+k \ (\text{mod} \ p)$, or 
$m=p$ and $k = p-j$.

Recall that the condition that 
$(x,y) \in R_\ell$ is equivalent to the 
condition that $x-y \in D_\ell$.
Therefore, $(b,0) \in R_i$ and
\begin{align*}
\rho_{m,k}^i 
&=|\{z \in GF(p)^n \ | \ 
\text{$(b,z) \in R_m$ and $(z,0) \in R_k$}\}|\\
&=| \{z \in GF(p)^n \ | \ 
\text{$b-z \in D_m$ and $z \in D_k$}\}|.
\end{align*}
Let $x = -z$. Since $f$ is even, $z \in D_k$ if and only 
if $-z \in D_k$.
Therefore, 
\begin{equation*}
\rho_{m,k}^i 
=| \{x \in GF(p)^n \ | \ 
\text{$b+x \in D_m$ and $x \in D_k$}\}|.
\end{equation*}
Summing over all pairs of indices $(m, k)$ such that 
$m-k \ (\text{mod} \ p) = j$, 
we obtain the total number of $x$ in $GF(p)^n$ for which
$f(x+b)-f(x)=j$:
	\[
 	 \sum_{k=0}^{p} \rho_{j+k \ (\text{mod} \ p), k}^i  +\rho_{p,p-j}^i  .
	\]
\hfill 
\end{proof}

From Proposition \ref{propn:balanced}, we 
obtain the following criterion for a function to be 
bent, in terms of structure constants of 
an association scheme.

\begin{proposition}
\label{propn:structureconstantcriterion} 
Let $f \colon GF(p)^n \rightarrow GF(p)$ be an even
function with $f(0)=0$ that 
determines an association scheme 
with structure constants $\rho_{ij}^k$. 
Then $f$ is bent if 
and only if 
\[
\left( \sum_{k=0}^{p} \rho_{j+k \ (\text{mod} \ p), k}^i \right) 
     + \rho_{p,p-j}^i
     =p^{n-1}
\]
for $1 \leq i \leq p$ and $1 \leq j \leq p-1$.
\end{proposition}

\subsection{Ito--Munemasa--Yamada theorem}
If a $p$-ary function determines 
an amorphic association scheme, 
we can use a theorem of Ito, Munemasa, and Yamada \cite{IMY91}
(which is formulated in more modern notation 
by Van Dam and Muzychuk in \cite[Corollary~1]{vDM10}) 
to obtain the structure constants 
of this amorphic association scheme.

Let $f \colon GF(p)^{2m} \rightarrow GF(p)$
be an even function with $f(0)=0$, where $p$ 
is a prime number greater than 2, and $m$ is a positive integer.
Suppose that $f$ determines an amorphic association scheme.
By the theorem of Gol'fand, Ivanov, and Klin
\cite{GIK94} mentioned in
Section \ref{subsec:vDam},
the component Cayley graphs $\Gamma_i$ of $f$ are  
strongly regular, 
and are either all of Latin square type, 
or all of negative Latin square type. 
Therefore, each graph $\Gamma_i$ has parameters of the form 
\begin{equation*}
(\nu,k,\lambda, \mu)=(N^2, (N-1)r_i, N+r_i^2-3r_i, r_i^2-r_i),
\end{equation*}
where $N$ equals $p^m$ in the Latin square case 
and $-p^m$ in the negative Latin square case, 
and each $r_i$ is an integer 
with the same sign as $N$.

The structure constants $\rho_{jk}^i$ of an association scheme 
 satisfy $\rho_{jk}^i=\rho_{kj}^i$. 
 Also, $\rho_{0j}^i=\delta_{ij}$ for $1 \leq i,j \leq p$, where $\delta_{ij}=0$
 if $i \neq j$, and $\delta_{ij}=1$ if $i=j$.
We will use the following theorem of 
\cite{IMY91}
to obtain the remaining structure constants.

\begin{theorem}[Ito, Munemasa, Yamada]
\label{thm:IMY}
If $f$ determines an amorphic association scheme, 
then the intersection numbers of the scheme, in the notation above, satisfy 
\begin{itemize}
\item[(a)]
	$\rho_{ii}^i = N +r_i^2-3r_i$ if $1 \leq i \leq p$, 
\item[(b)]
	$\rho_{jj}^i = (r_j-1)r_j$ if $i$ and $j$ are distinct and 
	$1 \leq i,j \leq p$,
 \item[(c)]
	 $\rho_{ij}^i=\rho_{ji}^i=(r_i-1)r_j$ if $i$ and $j$ are distinct and 
	 $1 \leq i,j \leq p$, and
 \item[(d)]
	 $\rho_{jk}^i=\rho_{kj}^i=r_jr_k$ if $i$, $j$, and $k$ are distinct 
	 and $1 \leq i,j,k \leq p$.
 \end{itemize}
 \end{theorem}
 
The proof below is included for completeness, 
and is only for the special case of interest to us, 
when the component Cayley graphs 
of $f$ are of feasible degrees.
In this case, 
$r_i = \frac N p$ for $1 \leq i \leq p-1$, 
and $r_p = \frac N p +1$.
 
 \begin{proof}
When the component Cayley graphs of $f$ 
 are of feasible degrees, these 
 structure constant formulas 
 may be derived from Proposition 
 \ref{propn:distinguished-eigenvalue} 
 and Equation (\ref{eqn:rhoijkformula}). 
 We note that 
 \[
 tr(A(i)A(j)A(k)) = 
 \sum_{x \in GF(p)^n} 
 \lambda_i(x) \lambda_j(x) \lambda_k(x),
 \]
 where $\lambda_i(x)$ is the eigenvalue 
 of $\Gamma_i$ corresponding to the 
 Hadamard vector $h(x)$ (see \S \ref{subsec:Had}).
 For $x \neq 0$, each eigenvalue 
 $\lambda_i(x)$ is either of the form 
 $N - r_i$ or $-r_i$. 
 For each fixed $x \neq 0$, there is 
 exactly one value of $i$ such that 
 $\lambda_i(x)$ is of the form $N-r_i$.
 Furthermore, the multiplicities of 
 each eigenvalue are also known.
 Let $i$, $j$, and $k$ be distinct integers 
 such that $1 \leq i, j, k \leq p$.
 By a straightforward counting argument, 
 we find that
 \[
 tr\left(A(i)^3\right) = N^2 (N-1)r_i (N+r_i^2-3r_i),
 \]
 \[
 tr\left(A(i)^2 A(j) \right) = 
 N^2(N-1) r_i(r_i-1)r_j,
 \]
 and 
 \[
 tr(A(i) A(j) A(k)) 
 = N^2 (N-1)r_i r_j r_k.
 \]
 Substituting these three cases 
 into Equation (\ref{eqn:rhoijkformula}), 
 we obtain the desired structure constants.  
 \hfill
 \end{proof}

\subsection{Proof that feasible amorphic functions are bent}
In this section, we use our structure constant 
criterion for bent functions to prove that 
if an even $p$-ary function of $2m$ variables, 
vanishing at 0,   
with component Cayley graphs of feasible degrees is amorphic, 
then it is bent.
As usual, the feasible degrees are those 
specified in \S \ref{subsec:feasible-degrees}.

\begin{theorem}
\label{thm:amorphic-bent}
Suppose that $f \colon GF(p)^{2m} \rightarrow GF(p)$ is an 
even function such that $f(0)=0$, where $p$ 
is a prime number greater than 2, and $m$ is a positive integer, 
such that the component Cayley graphs of $f$ 
are of feasible degrees and $f$ determines 
an amorphic association scheme.
Then $f$ is bent.
\end{theorem}

\begin{remark}
{\rm
Theorem \ref{thm:amorphic-bent} 
is equivalent to Theorem \ref{thm:main1}, 
due to the results of van Dam \cite{vD03} and
Gol'fand, Ivanov, and Klin \cite{GIK94} 
discussed 
in \S \ref{subsec:vDam}, 
but we include both theorems because the proofs are different.
}
\end{remark}

\begin{proof}
We will use the structure constant criterion 
of Proposition \ref{propn:structureconstantcriterion} 
to show that $f$ is bent: we will show that 
\begin{equation}
\label{eqn:proofsum}
\left( \sum_{k=0}^{p} \rho_{t+k \ (\text{mod} \ p), k}^s \right) + \rho_{p,p-t}^s = p^{2m-1}
\end{equation}
for $1 \leq s \leq p$ and $1 \leq t \leq p-1$.
To evaluate this sum, we 
use the theorem of Ito, Munemasa, and Yamada
(Theorem \ref{thm:IMY}).
We consider the following four cases: 
$s=p$;
$s \neq p$ and $s=t$;
$s \neq p$ and $s = p-t$; and 
$s \neq p$, $s \neq t$, and $s \neq p-t$ (this last
case cannot occur for $p=3$).

Terms of types (a) and (b) in Theorem
\ref{thm:IMY} do not occur in the sum of 
Equation (\ref{eqn:proofsum}). 
In the following chart, we indicate how many times 
terms of each of the remaining types occur in the sum, 
in each of the four cases. 
There are a total of $p+2$ terms in each column.
In the chart, we use the convention that 
$i$, $j$, and $k$ represent 
distinct values in the 
set $\{1, 2, \dots , p-1\}$. 
We also use the notation 
$r = \frac N p$.

\vspace{0.1in}
\begin{tabular}{|l|c|c|c|c|}
\hline 
 &\multicolumn {4} {c|} { Number of occurrences} \\
\hline
Term type 
 & \multirow{2}{*}{$s=p$} 
 & \multirow{2}{*}{$s =t \neq p$} 
 & \multirow{2}{*}{$s =p-t \neq p$}  &
 $s \neq p, t, p-t$ \\
 and value & & & & $p \neq 3$\\
\hline
$\rho_{0i}^i =\rho_{i0}^i=1$  & 0 & 1 & 1 & 0\\
\hline
$\rho_{0j}^i=\rho_{j0}^i=0$  & 0 & 1 & 1 & 2\\
\hline
$\rho_{0i}^p=\rho_{i0}^p=0$  & 2 & 0 & 0 & 0\\
\hline
$\rho_{ip}^p=\rho_{pi}^p=r^2$  & 2 & 0 & 0 & 0\\
\hline
$\rho_{ij}^i=\rho_{ji}^i=r^2-r$  & 0 & 1 & 1 & 2\\
\hline
$\rho_{pi}^i=\rho_{ip}^i=r^2-1$  & 0 & 1 & 1 & 0\\
\hline
$\rho_{jp}^i=\rho_{pj}^i=r^2+r$ & 0 & 1 & 1 & 2\\
\hline
$\rho_{jk}^i=\rho_{kj}^i=r^2$ & $0$ & $p-3$& $p-3$ & $p-4$ \\
\hline
$\rho_{ij}^p=\rho_{ji}^p=r^2$ & $p-2$ & $0$& $0$ & $0$ \\
\hline
\end{tabular}
\vspace{0.1in}

We see that for each column, the number of occurrences of each term type multiplied by the value sums to $pr^2 = p^{2m-1}$.  
For example, for the $s=p$ column, we obtain $2r^2+(p-2)r^2 = pr^2$. 
Therefore, $f$ is bent.
\hfill 
\end{proof}

\subsection{Examples}
To illustrate some properties of amorphic bent 
functions, we include an example from
\cite{CJMPW16}.  
Additional properties of the functions 
in this example are given in 
Examples \ref{ex:GF32fa} and \ref{ex:GF32fb}.
We show a sample sum 
of structure constants of the type 
of Propositions \ref{propn:balanced} 
and \ref{propn:structureconstantcriterion}.
We also give an example of a bent function 
which is not amorphic.

\begin{example} (\cite{CJMPW16})
\label{ex:gf32bent}
Let $f \colon GF(3)^2\to GF(3)$ be an even bent function with $f(0)=0$.
Then $f(x_0,x_1)$ is equivalent to either 
$-x_0^2+x_1^2$ 
or $x_0^2+x_1^2$ under
the action of $GL(2,GF(3))$ on $(x_0,x_1)$.
In each case, the function $f$ 
determines an amorphic association scheme.

\begin{enumerate}
\item
In the case of $-x_0^2+x_1^2$,
the structure constants 
$\rho_{ij}^k$ are given in the following arrays.

\[
\begin{array}{cc}
\begin{array}{c|cccc}
\rho_{ij}^0 & 0 & 1 & 2 & 3 \\ \hline
0         & 1 & 0 & 0 & 0 \\
1         & 0 & 2 & 0 & 0 \\
2         & 0 & 0 & 2 & 0 \\
3         & 0 & 0 & 0 & 4 \\
\end{array}
 &
\begin{array}{c|cccc}
\rho_{ij}^1 & 0 & 1 & 2 & 3 \\ \hline
0         & 0 & 1 & 0 & 0 \\
1         & 1 & 1 & 0 & 0 \\
2         & 0 & 0 & 0 & 2 \\
3         & 0 & 0 & 2 & 2 \\
\end{array} \\
 & \\
\begin{array}{c|cccc}
\rho_{ij}^2 & 0 & 1 & 2 & 3 \\ \hline
0         & 0 & 0 & 1 & 0 \\
1         & 0 & 0 & 0 & 2 \\
2         & 1 & 0 & 1 & 0 \\
3         & 0 & 2 & 0 & 2 \\
\end{array}
 &
\begin{array}{c|cccc}
\rho_{ij}^3 & 0 & 1 & 2 & 3 \\ \hline
0         & 0 & 0 & 0 & 1 \\
1         & 0 & 0 & 1 & 1 \\
2         & 0 & 1 & 0 & 1 \\
3         & 1 & 1 & 1 & 1 \\
\end{array} \\
\end{array}
\]
\item
In the case of $x_0^2+x_1^2$, 
the component Cayley graph $\Gamma_3$ is 
empty, and the
structure constants
$\rho_{ij}^k$ are given in the following arrays.

\[
\begin{array}{cc}
\begin{array}{c|ccc}
\rho_{ij}^0 & 0 & 1 & 2  \\ \hline
0         & 1 & 0 & 0  \\
1         & 0 & 4 & 0  \\
2         & 0 & 0 & 4  \\
\end{array}
 &
\begin{array}{c|cccc}
\rho_{ij}^1 & 0 & 1 & 2  \\ \hline
0         & 0 & 1 & 0  \\
1         & 1 & 1 & 2  \\
2         & 0 & 2 & 2  \\
\end{array} \\
 & \\
\begin{array}{c|cccc}
\rho_{ij}^2 & 0 & 1 & 2  \\ \hline
0         & 0 & 0 & 1  \\
1         & 0 & 2 & 2  \\
2         & 1 & 2 & 1  \\
\end{array}
 &
{\rm no} \  \rho_{ij}^3 \\
\end{array}
\]
\end{enumerate}
\end{example}

\begin{example}
{\rm
Consider the structure constants for 
$-x_0^2+x_1^2$ given above.
The expression in Proposition \ref{propn:balanced},
for $i=1$ and $j=2$, is
	\[
	 \rho^1_{2,0}+\rho^1_{0,1}+\rho^1_{1,2}+\rho^1_{2,3}
	 	 +\rho^1_{3,1}=3.
	\]
The other sums from Proposition \ref{propn:balanced}
can be calculated from the arrays above in a similar manner.
}
\end{example}

\begin{example}
\label{ex:GF52nonexample}
{\rm
Consider the function 
$f \colon GF(5)^2 \to  GF(5)$
given by 
\[
f(x_0,x_1)=3x_0^4+2x_0^2+2x_0 x_1.
\]
It can be checked that $f$ is bent with
\begin{align*}
D_1&=\{ (1, 3), (2, 0), (3, 0), (4, 2)\}\\
D_2&=\{(1, 1), (2, 4), (3, 1), (4, 4)\}\\
D_3&=\{(1, 4), (2, 3), (3, 2), (4, 1)\}\\
D_4&=\{(1, 2), (2, 2), (3, 3), (4, 3)\}\\
D_5&=\{(0, 1), (0, 2), (0, 3), (0, 4), (1, 0), (2, 1), (3, 4), (4, 0)\}
\end{align*}
However, graphs $\Gamma_1$, $\Gamma_2$, and $\Gamma_4$ have 
6 distinct eigenvalues, $\Gamma_3$ has 2 distinct eigenvalues, and $\Gamma_5$ has 
7 distinct eigenvalues, so the graphs are 
not all strongly regular. 
By the theorem of Gol'fand, Ivanov, and Klin
\cite{GIK94} (see \S \ref{subsec:vDam}),
$f$ is not amorphic.
}
\end{example}

We give more examples in \S \ref{sec:manyexamples}.

\section{Orthogonal arrays and bent functions}
Orthogonal arrays are closely related to strongly regular graphs 
of Latin square type, and may be 
used to construct amorphic association schemes.
See \cite[\S 10.4]{GR01}
and \cite{vDM10} for further background on 
these topics.
We will describe how to construct 
bent functions using orthogonal arrays. 

\subsection{Orthogonal arrays}
Let $S$ be a set of size $N$. 
An {\it orthogonal array of size $r \times N^2$ 
with entries in $S$} consists of 
$r$ rows of $N^2$ entries from $S$, 
such that for any two rows, the $N^2$ ordered pairs 
determined by the columns are all distinct.
Such an array is denoted $OA(r,N)$.

We are primarily interested in orthogonal arrays of size 
$(N+1) \times N^2$, where $S=GF(p)^m$, $N=p^m$, 
and $p$ is prime.

\begin{example}
\label{ex:OAGF32}
{\rm
Let $S=GF(3)$, $r=4$, and $N=3$. The following is an 
$OA(4,3)$ with entries in $S$.
\begin{equation*}
\mathcal O \mathcal A = 
\begin{array}{c c c c c c c c c c c}
0&0&0&\ &1&1&1&\ &2&2&2\\
0&1&2&\ &0&1&2&\ &0&1&2\\
0&1&2&\ &1&2&0&\ &2&0&1\\
0&1&2&\ &2&0&1&\ &1&2&0
\end{array}
\end{equation*}
}
\end{example}

\subsection{Latin square type graphs from orthogonal arrays}
We can form a graph $\Gamma$ from an orthogonal array $OA(r,N)$
as follows.
The vertices of $\Gamma$ are the columns of the array.
Two distinct vertices are connected by an edge exactly when the 
columns have the same entry in one row. 
It is well-known that the graph $\Gamma$ is either complete 
(in the case $r=N+1$) 
or strongly regular of Latin square type.
We include a proof for the convenience of the reader.
We then give examples in which we construct graphs 
from orthogonal arrays 
and use these graphs to construct 
amorphic bent functions.

\begin{lemma}
The graph $\Gamma$ determined by an $OA(r,N)$ is 
either complete or 
strongly regular of Latin square type with parameters 
\[
\left( N^2, (N-1)r, N+r^2-3r,r^2-r \right).
\]
\end{lemma}

\begin{proof}
Consider any column $v$ of the array. 
The $i$th entry of $v$ occurs in exactly $N-1$ other 
locations in row $i$.
By the definition of an orthogonal array, two columns can 
agree in at most one row.
Thus, vertex $v$ has $(N-1)r$ neighbors.

If $r=N+1$, then each vertex $v$ has $N^2-1$ neighbors, 
and $\Gamma$ is complete. 

Suppose that $v$ and $w$ are neighbors, i.e., 
$v$ and $w$ are distinct columns, 
with equal entries in some row $i$.
A column $u$ which is a neighbor of both $v$ and $w$ 
agrees with each of $v$ and $w$ in exactly one row.
There are $N-2$ neighbors $u$ that 
have the same entry in row $i$ 
as $v$ and $w$. 
Any other neighbor $u$ of  
$v$ and $w$ 
must agree with $v$ in some row $j\neq i$ 
and with $w$ in some row $k\neq i,j$. 
There are $(r-1)(r-2)$ such ordered pairs $(j,k)$, 
each corresponding 
to exactly one neighbor $u$ of $v$ and $w$.
Thus there are $N+r^2-3r$ common neighbors of 
$v$ and $w$. 

Finally, suppose that $\Gamma$ is not complete, 
and that $v$ and $w$ are distinct columns which are not adjacent.
A neighbor $u$ of $v$ and $w$ must agree with $v$ in some row $j$ 
and with $w$ in some row $k \neq j$. 
There are $r(r-1)$ such ordered pairs $(j,k)$, each corresponding 
to exactly one neighbor $u$ of $v$ and $w$.
\hfill
\end{proof}

\begin{example}
\label{ex:K9}
{\rm
The graph $\Gamma$ determined by the orthogonal array 
$\mathcal O \mathcal A = OA(4,3)$ of 
Example \ref{ex:OAGF32} is a complete graph on 9 vertices.
}
\end{example}

\begin{example}
\label{ex:negpos}
{\rm
We may partition the orthogonal array $\mathcal O \mathcal A$ of Example \ref{ex:OAGF32} into 
two smaller orthogonal arrays:
\begin{equation*}
\mathcal O \mathcal A_1 \ = \ 
\begin{array}{c c c c c c c c c c c}
0&0&0&\ &1&1&1&\ &2&2&2\\
0&1&2&\ &0&1&2&\ &0&1&2
\end{array}
\end{equation*}
and
\begin{equation*}
\mathcal O \mathcal A_2 \ = \ 
\begin{array}{c c c c c c c c c c c}
0&1&2&\ &1&2&0&\ &2&0&1\\
0&1&2&\ &2&0&1&\ &1&2&0
\end{array}
.
\end{equation*}

We identify the vertices of 
the graphs $\Gamma_1$ and $\Gamma_2$ 
corresponding to $\mathcal O \mathcal A_1$ 
and $\mathcal O \mathcal A_2$
with the vertices of the graph $\Gamma$ 
corresponding to $\mathcal O \mathcal A$ 
in the obvious way.
The graphs $\Gamma_1$ and $\Gamma_2$ are 
both strongly regular with parameters
$( 9, 4, 1, 2 )$.
They are of Latin square type with $N=3$, $r=2$, 
and of negative Latin square type with $N=-3$, $r=-1$.
They form a strongly regular
decomposition of the complete graph 
on $9$ vertices. 
The negative Latin square type 
decomposition is feasible.
By van Dam's theorem 
(see \S \ref{subsec:vDam}), 
this graph decomposition determines an amorphic 
association scheme. 

Let us use the first two entries in each column 
of $\mathcal O \mathcal A$
to identify the 9 vertices of these graphs 
with elements of $GF(3)^2$. 
The neighbors of $(0,0)$ in $\Gamma_1$ form the set
\[ 
D_1 = \{(0,1), (0,2), (1,0), (2,0)\}.
\]
The neighbors of $(0,0)$ in $\Gamma_2$ form the set
\[ 
D_2 = \{ (1,1), (2,2), (1,2), (2,1) \}.
\]
Now, let us define an even function 
$f \colon GF(3)^2 \to GF(3)$ by setting 
\[
f(x) = 
\begin{cases}
0 &\qquad \text{if $x=(0,0)$,}\\
1 & \qquad \text{if $x \in D_1$,}\\
2 & \qquad \text{if $x \in D_2$.}
\end{cases}
\]
The graphs $\Gamma_1$ and $\Gamma_2$ are  
component Cayley graphs of $f$. The third 
component Cayley graph of $f$ is empty.
The function $f$ is amorphic, 
and consequently bent.
It can be shown that $f$ is given by 
\[
f(x_0,x_1) = x_0^2+x_1^2.
\]
See Example 
\ref{ex:GF32fb} for more properties of 
this function.
}
\end{example}

\begin{example}
{\rm
Similarly, if we partition the orthogonal array $\mathcal O \mathcal A$ of Example \ref{ex:OAGF32}
into three arrays, consisting of the first row, 
the second row, and the last two rows, we obtain sets 
$D_1 = \{ (0,1), (0,2)\}$, $D_2=\{ (1,0),(2,0)\}$, 
and $D_3 = \{ (1,1),(2,2), (1,2), (2,1)\}$. 
The corresponding graph decomposition of the 
complete graph on $9$ vertices is a 
strongly regular decomposition.
All three graphs are of Latin square type.
The graphs $\Gamma_1$ and $\Gamma_2$ 
have parameters $(9,2,1,0)$, and 
the graph $\Gamma_3$ has parameters $(9,4,1,2)$.
The sets $D_1$, $D_2$, and $D_3$
determine the amorphic bent function given by
\[
g(x_0,x_1) = -x_0^2+x_1^2.
\]
See Example 
\ref{ex:GF32fa} for more properties of 
this function.
}
\end{example}

\subsection{Bent functions from orthogonal arrays}
In this section, we show how to construct 
an amorphic bent $p$-ary function 
of Latin square type on $2m$ variables, 
for any prime number $p$ greater than 2 and 
any positive integer $m$. 
We start with an orthogonal array of appropriate 
dimensions, and use a generalization of the procedure 
in the previous example.
The technique for constructing amorphic association 
schemes from orthogonal arrays is known 
(see \cite[\S 5]{vDM10}), but we include 
details in the case of interest for completeness.

Thanks to a construction of Bush \cite{Bu52}, it is known that 
it is possible to construct an 
orthogonal array of type $OA(N+1,N)$ 
when $N = p^m$, for every prime number $p$.
If the entries are in $GF(p)^m$, we 
may construct our orthogonal array in such a 
way that  
all entries in the first column are equal to 
the 0 element of $GF(p)^m$.

Let $p$ be a prime number greater than $2$.
Consider a partition of the $N+1$ rows of an 
$OA(N+1,N)$ with entries in $GF(p)^m$ into 
$p-1$ sets of $r = \frac N p$ rows, and 
one set of $r_p=\frac N p +1$ rows. 
We denote the corresponding orthogonal 
subarrays by $\mathcal O \mathcal A_1, \mathcal O 
\mathcal A_2, \ldots , \mathcal O \mathcal A_p$.
This partition determines a strongly regular 
decomposition $\Gamma_1, \Gamma_2, \dots , \Gamma_p$
of the complete graph on 
$p^{2m}$ vertices, consisting of $p$ strongly regular 
graphs of Latin square type. 
(Once again, we identify the vertices of each 
graph $\Gamma_i$ with the vertices of the 
complete graph on $GF(p)^{2m}$ determined 
by the original $OA(N+1,N)$.)
By van Dam's theorem \cite{vD03} 
(see \S \ref{subsec:vDam}), 
this graph decomposition determines an amorphic 
association scheme. 

Let $D_i$ be the set of all neighbors of $0$ in 
the graph $\Gamma_i$ corresponding to $\mathcal O \mathcal A_i$.
By our assumption that the first column of our array 
consists of 0 entries, 
a vertex $v$ is in $D_i$ if and only if it has a $0$ 
entry in one of the rows in $\mathcal O \mathcal A_i$.
Therefore, $D_i$ is symmetric, i.e., 
if $x \in D_i$ then $-x \in D_i$.
Define an even function 
$f \colon GF(p)^{2m} \to GF(p)$ by setting 
\[
f(x) = 
\begin{cases}
0 &\qquad \text{if $x=0$,}\\
i & \qquad \text{if $x \in D_i$ and $1 \leq i \leq p-1$,}\\
0 & \qquad \text{if $x \in D_p$.}
\end{cases}
\]
Then the component Cayley graphs of $f$ are 
the strongly regular Latin square type graphs $\Gamma_i$.
By Theorem \ref{thm:main1} or 
Theorem \ref{thm:amorphic-bent}, the function $f$ is 
amorphic and bent.

\section{Examples}
\label{sec:manyexamples}
In this section we give examples of amorphic 
bent functions of Latin and negative Latin square type, 
together with their duals. We also give 
examples of 5-ary bent functions whose component 
Cayley graphs are not all strongly regular.

\subsection{Examples constructed from orthogonal arrays}
In this section we provide three examples of 
amorphic bent $p$-ary functions of Latin square type, 
which were constructed from orthogonal arrays 
using a computer.
In each case, an orthogonal array of size
$(p^2+1) \times p^4$ was constructed by the method of Bush, 
with symbols in $GF(p^2)$, 
which were then replaced by entries from 
$GF(p)^2$. 
The algebraic form of the resulting function was 
found using \cite[Theorem 53 and Corollary 6]{CJMPW16}.

\begin{example}
\label{ex:OA3LST}
{\rm
The following amorphic $3$-ary bent function on $GF(3)^4$ 
was constructed 
from a $10 \times 81$ orthogonal array, using a computer:
\[
f(x_0,x_1,x_2,x_3)
=2 x_0 x_3 + x_1 x_2 + x_0^2 x_1 x_2 + 2 x_0 x_1^2 x_3.
\]
The component Cayley graphs $\Gamma_i$ of $f$ are all strongly regular 
of Latin square type. The parameters of $\Gamma_1$ and 
$\Gamma_2$ are $(81,24,9,6)$ and the parameters of $\Gamma_3$ 
are $(81,32,13,12)$.
The function $f$ is regular with dual 
\[
f^*(x_0,x_1,x_2,x_3)=x_0 x_3+2x_1 x_2+x_0 x_2^2 x_3 + 2x_1 x_2 x_3^2.
\]
}
\end{example}

\begin{example}
\label{ex:OA5LST}
{\rm
The following amorphic $5$-ary bent function on $GF(5)^4$ 
was constructed 
from an orthogonal array, using a computer:
\[
4 x_0^3  x_3  + 3 x_0^2  x_1   x_2  + 
 x_0   x_1^2  x_3  + 3 x_1^3  x_2  + 
  x_0^4  x_1^3  x_2  +3 x_0^3  x_1^4  x_3 
 \]
The component Cayley graphs $\Gamma_i$ are all strongly regular 
of Latin square type. The parameters of $\Gamma_i$, 
for $1 \leq i \leq 4$, 
are $(625,120,35,20)$ and the parameters of $\Gamma_5$ 
are $(625,144,43,30)$.
The function is regular with dual 
\[
2 x_1 x_2^3 + 4 x_0 x_2^2 x_3 + 2 x_1 x_2 x_3^2 
+ x_0 x_3^3 + 2 x_0 x_2^4 x_3^3 + 
 4 x_1 x_2^3 x_3^4.
\]
 }
\end{example}

\begin{example}
\label{ex:OA7LST}
{\rm
The following amorphic $7$-ary bent function on $GF(7)^4$ 
was constructed 
from an orthogonal array, using a computer:
\[
6  x_0^5  x_3  + 4  x_0^4  x_1   x_2  + 
x_0^3  x_1^2  x_3  + 6  x_0^2  x_1^3  x_2  + 
5  x_0   x_1^4  x_3  + 4  x_1^5  x_2  + 
 5  x_0^6  x_1^5  x_2  +  4  x_0^5  x_1^6  x_3 
 \]
 The component Cayley graphs $\Gamma_i$ are all strongly regular 
of Latin square type. The parameters of $\Gamma_i$, 
for $1 \leq i \leq 6$, 
are $(2401, 336, 77, 42)$ and the parameters of $\Gamma_7$ 
are $(2401, 384, 89, 56)$.
The function is regular with dual 
\[
2 x_0 x_2^4 x_3 +6 x_0 x_2^2 x_3^3 + 
x_0 x_3^5 + 3 x_1 x_2^5 + 
x_1 x_2^3 x_3^2 +  3 x_1 x_2 x_3^4 + 
 3 x_0 x_2^6 x_3^5 + 2 x_1 x_2^5 x_3^6.
\]
 }
\end{example}

\subsection{Examples on $GF(3)^2$}

Every even bent function $f \colon GF(3)^2 \rightarrow GF(3)$ 
with $f(0)=0$ is equivalent to the function of 
Example \ref{ex:GF32fa} or of 
Example \ref{ex:GF32fb} below under 
the action of $GL(2,GF(3))$ on the variables $(x_0,x_1)$ 
(see \cite[Proposition 10]{CJMPW16}).
Example \ref{ex:GF32fa} is amorphic of Latin square type, 
and Example \ref{ex:GF32fb} is amorphic of 
negative Latin square type.

\begin{example}
\label{ex:GF32fa}
{\rm
The function $f \colon GF(3)^2 \rightarrow GF(3)$ given by
\[
f(x_0,x_1)=-x_0^2+x_1^2
\]
is an even bent function with $f(0)=0$.
The component Cayley graphs $\Gamma_1$ and $\Gamma_2$ are strongly regular 
with parameters
$(9,2,1,0)$.
The component Cayley graph $\Gamma_3$ is 
strongly regular with parameters $(9,4,1,2)$.
The graphs $\Gamma_1$, $\Gamma_2$, and $\Gamma_3$ are 
all of Latin square type.  
The function $f$ is amorphic and regular, with 
dual $f^*(x_0,x_1)=x_0^2-x_1^2$.
}
\end{example}

\begin{example}
\label{ex:GF32fb}
{\rm
The function $f \colon GF(3)^2 \rightarrow GF(3)$ given by
\[
f(x_0,x_1)=x_0^2+x_1^2
\]
is an even bent function with $f(0)=0$.
The component Cayley graphs $\Gamma_1$ and $\Gamma_2$ are strongly regular 
with parameters
$(9,4,1,2)$.
The component Cayley graph $\Gamma_3$ is empty.
The graphs $\Gamma_1$ and $\Gamma_2$ are of Latin square 
type with $N=3$ and $r=2$ and of negative Latin square type 
with $N=-3$ and $r=-1$.  
However, only the negative Latin square type parameters 
satisfy the feasibility condition $r = \frac N p$, where $p=3$. 
The function $f$ is amorphic and (-1)-weakly regular, with 
dual $f^*(x_0,x_1)=-x_0^2-x_1^2$.
}
\end{example}

\subsection{Examples on $GF(5)^2$}
\label{subsec:GF52}
In a previous paper, 
\cite[Proposition 14 and Example 65]{CJMPW16}, 
we classified all 
even bent functions
$g \colon GF(5)^2 \rightarrow GF(5)$ 
with $g(0)=0$
into eleven equivalence classes 
under the action of $GL(2,GF(5))$ on 
the variables $(x_0,x_1)$.
Using additional computer calculations, 
it can be shown that the three functions 
of Example \ref{ex:GF52g}
represent the only equivalence classes 
whose functions are of amorphic Latin 
square type, i.e., those
whose component Cayley graphs are 
all strongly regular of Latin square type.
It can also be shown that there is 
no even bent function $g \colon GF(5)^2 \rightarrow GF(5)$ 
with $g(0)=0$ whose component Cayley graphs 
are all strongly regular of feasible negative Latin square type
(see Remark \ref{rk:no_NLST_m1p5}).
The remaining examples in this section 
are not amorphic. 
In these examples, some or all of the 
component Cayley graphs are not strongly regular.

\begin{example}
\label{ex:GF52g}
{\rm
The following three functions 
$g_i \colon GF(5)^2 \rightarrow GF(5)$ 
are even bent functions with $g_i(0)=0$:
\[
g_1(x_0,x_1)=x_0^3x_1+2x_1^4, \
g_2(x_0,x_1)=-x_0x_1^3+x_1^4, \
g_3(x_0,x_1)=-x_0^3x_1+x_1^4.
\]
The component Cayley graphs $\Gamma_i$,
for $1 \leq i \leq 4$,
are strongly regular 
with parameters
$(25,4,3,0)$.
The component Cayley graph $\Gamma_5$ is 
strongly regular with parameters $(25,8,3,2)$.
The graphs $\Gamma_i$, for $1 \leq i \leq 5$,
are all of Latin square type.  
The functions $g_i$ are amorphic and regular, with 
duals 
\[
g_1^*(x_0,x_1)=2x_0^4-x_0 x_1^3, \
g_2^*(x_0,x_1)=x_0^4+x_0^3 x_1, \
g_3^*(x_0,x_1)=x_0^4+x_0 x_1^3.
\]
}
\end{example}

\begin{example}
\label{ex:counter1}
{\rm
The function $g \colon GF(5)^2 \rightarrow GF(5)$ given by
\[
g(x_0,x_1)=-x_0^2 + 2x_1^2
\]
is an even bent function with $g(0)=0$. 
The function $g$ is (-1)-weakly regular, with 
dual $g^*(x_0,x_1) = -x_0^2 + 3x_1^2$.
The degree of $\Gamma_i$, for $1 \leq i \leq 4$, is $k_i=6$.
The graph $\Gamma_5$ is empty.
It can be shown, by 
checking the number of distinct eigenvalues 
of each graph, that the component Cayley 
graphs $\Gamma_i$ are not strongly regular.
The unions $\Gamma_1 \cup \Gamma_4$ and 
$\Gamma_2 \cup \Gamma_3$ are 
strongly regular of negative Latin square 
type with parameters $(25,12,5,6)$.
}
\end{example}

\begin{example}
\label{ex:counter2}
{\rm
The function $g \colon GF(5)^2 \rightarrow GF(5)$ given by
\[
g(x_0,x_1)=-x_0 x_1 + x_1^2
\]
is an even bent function with $g(0)=0$. 
The function $g$ is regular, with 
dual $g^*(x_0,x_1) = x_0^2+x_0x_1$.
The graph $\Gamma_5$ is strongly regular 
of Latin square type with parameters $(25, 8, 3, 2)$.
The degree of $\Gamma_i$, for $1 \leq i \leq 4$, is $k_i=4$.
It can be shown, by 
checking the number of distinct eigenvalues 
of each graph, that the component Cayley 
graphs $\Gamma_1$, $\Gamma_2$, $\Gamma_3$, 
and $\Gamma_4$ are not strongly regular.
The unions $\Gamma_1 \cup \Gamma_4$ and 
$\Gamma_2 \cup \Gamma_3$ are 
strongly regular of Latin square 
type with parameters $(25, 8, 3, 2)$.
}
\end{example}

\begin{example}
\label{ex:counter3}
{\rm
The function $g \colon GF(5)^2 \rightarrow GF(5)$ given by
\[
g(x_0,x_1)=2x_0 x_1^3+x_1^4-x_1^2
\]
is an even bent function with $g(0)=0$. 
The function $g$ is regular, with 
dual $g^*(x_0,x_1)=x_0^2 + x_0^4 + 3 x_0^3 x_1$.
None of the component Cayley graphs $\Gamma_i$ 
is strongly regular.
Moreover, no union of the component Cayley graphs 
$\Gamma_i \cup \Gamma_j$ for $i \neq j$
(and hence no union of the form 
$\Gamma_i \cup \Gamma_j \cup \Gamma_k$
for $i$, $j$, and $k$ distinct)
is strongly regular.
The degrees of the component Cayley graphs 
are $k_i = 4$, for $1 \leq i \leq 4$, 
and $k_5 = 8$.
}
\end{example}

\subsection{Examples on $GF(3)^4$}
Recall that in Example \ref{ex:OA3LST}
we gave a bent function on $GF(3)^4$ 
whose component Cayley graphs are 
all strongly regular of Latin square type. 
We now give two examples of bent 
functions on $GF(3)^4$ whose
component Cayley graphs are all 
strongly regular of negative 
Latin square type.

\begin{example}
{\rm
The function $f \colon GF(3)^4 \rightarrow GF(3)$ given by
\[
f(x_0,x_1,x_2,x_3)=-x_0^2-x_1^2+x_2 x_3
\]
is an even bent function with $f(0)=0$.
The component Cayley graphs $\Gamma_1$ and $\Gamma_2$ are strongly regular with parameters $(81,30,9,12)$.
The component Cayley graph $\Gamma_3$ is 
strongly regular with parameters $(81,20,1,6)$.
The graphs $\Gamma_1$, $\Gamma_2$, and $\Gamma_3$ are 
all of negative Latin square type.  
The function $f$ is amorphic and (-1)-weakly regular, 
with dual 
\[
f^*(x_0,x_1,x_2,x_3)=x_0^2+x_1^2-x_2 x_3.
\]
In this example, $D_1^*=D_2$, $D_2^*=D_1$, and 
$D_3^*=D_3$.
}
\end{example}

\begin{example}
{\rm
The function $f \colon GF(3)^4 \rightarrow GF(3)$ given by
\[
f(x_0,x_1,x_2,x_3) = x_0^2 + x_1^2 +x_0 x_2 + 2 x_2 x_3
\]
is an even bent function with $f(0)=0$.
The component Cayley graphs $\Gamma_1$ and $\Gamma_2$ are strongly regular with parameters $(81,30,9,12)$.
The component Cayley graph $\Gamma_3$ is 
strongly regular with parameters $(81,20,1,6)$.
The graphs $\Gamma_1$, $\Gamma_2$, and $\Gamma_3$ are 
all of negative Latin square type.
The function $f$ is amorphic and (-1)-weakly regular, 
with dual
\[
f^*(x_0,x_1,x_2,x_3) = 2 x_0^2 + 2 x_1^2 +x_0 x_3+  x_2 x_3+2x_3^2.
\]
In this example, we know of no simple relationship 
between the sets $D_1$, $D_2$, and $D_3$ and 
the sets $D_1^*$, $D_2^*$, and $D_3^*$.
}
\end{example}

\subsection{Ideas for further study}

We conclude with some questions and ideas for further study.

\begin{enumerate}
\item Can we find a way to construct all amorphic bent functions of
the type of Theorem \ref{thm:amorphic-bent}? 
Can we count them?
\item
Consider equivalence classes of $p$-ary functions under the 
action of $GL(n,GF(p))$ on coordinates.
Do there exist non-equivalent bent functions 
which determine isomorphic association schemes?
Which of our amorphic examples 
have isomorphic association schemes?
\item
Find examples of functions that are not bent, 
whose level sets determine association schemes 
that are not amorphic. 
\end{enumerate}

\end{document}